\documentclass[12pt]{article}
\usepackage{a4wide}
\usepackage{latexsym, amsmath, amsfonts, amsthm, amssymb}
\usepackage{times}

\newtheorem{theo}{Theorem}
\newtheorem{prop}[theo]{Proposition}
\newtheorem{lemma}[theo]{Lemma}

\newtheorem{rem}[theo]{Remark}

\newtheorem{defi}{Definition}

\newcommand{\R}{\ensuremath{\mathbb{R}}} 

\def\eps{\varepsilon}

\newcommand{\upchi}{\raise1pt\hbox{$\chi$}}

\newcommand{\be}{\begin{equation}}
\newcommand{\ee}{\end{equation}}
\def\benu{\begin{enumerate}}
\def\eenu{\end{enumerate}}

\def \RR {\mathbb R}
\def \R {\mathbb R}
\def \eps {\varepsilon}
\def \vphi {\varphi}

\def \cH {\mathcal H}

\def \cJ {\mathcal J}
\def \cR {\mathcal R}
\def \cI {\mathcal I}

\def\qed{\hfill $\vcenter{\hrule height .3mm
\hbox {\vrule width .3mm height 2.1mm \kern 2mm \vrule width .3mm
height 2.1mm} \hrule height .3mm}$ \bigskip}

\begin{document}

\title{Moment Measures}
\author{D. Cordero-Erausquin\textsuperscript{1} and B. Klartag\textsuperscript{2}}

\date{}

\footnotetext[1]{Institut de Math\'ematiques de Jussieu and Institut Universitaire de France,
Universit\'e Pierre et Marie Curie (Paris 6), 4 place Jussieu, 75252
Paris, France. Email: cordero@math.jussieu.fr}

\footnotetext[2]{School of Mathematical Sciences, Tel Aviv
University, Tel Aviv 69978, Israel. Email: klartagb@tau.ac.il}

\vspace{-2in}\maketitle

\begin{abstract}
With any convex function $\psi$ on a finite-dimensional linear space $X$ such that $\psi$ goes to $+\infty$ at infinity,
we associate a Borel measure $\mu$ on $X^*$. The measure $\mu$ is obtained by pushing forward the measure $e^{-\psi(x)} \, dx$ under the differential of $\psi$.
We propose a class of convex functions -- the essentially-continuous, convex functions -- for which the above correspondence
is in fact a bijection onto the class of   finite Borel measures whose barycenter is at the origin and whose support spans $X^*$.
The construction is related to toric K\"ahler-Einstein metrics in complex geometry, to Pr\'ekopa's inequality, and to the Minkowski problem in convex geometry.
\end{abstract}

\section{Introduction}

The aim of the present work is to extend the results  on moment
measures obtained by Berman and Berndtsson~\cite{BB} in their work
on K\"ahler-Einstein metrics in toric varieties, which builds upon earlier  works by Wang and Zhu \cite{WZ},
by Donaldson \cite{donaldson} and by E. Legendre \cite{e_legednre}. Simultaneously,
our analysis of moment measures should be viewed as a functional version
of the classical Minkowski problem (see, e.g., Schneider \cite[Section 7.1]{schneider})
or the logarithmic Minkowski problem of B\"or\"oczky, Lutwak, Yang and Zhang \cite{BLYZ}.
Yet a third point of view, is that we discuss a certain kind of Monge-Amp\`ere equation,
and establish existence and uniqueness of generalized solutions.

\medskip
Suppose that $\psi: \RR^n \rightarrow \RR \cup \{ + \infty \}$ is a convex function, i.e., for any $0 < \lambda < 1$
and $x, y \in \RR^n$,
$$ \psi \left( \lambda x + (1- \lambda) y \right) \leq \lambda \psi(x) + (1 - \lambda) \psi(y) $$
whenever $\psi(x) < +\infty$ and $\psi(y) < +\infty$. In this note, we treat $+\infty$ as a legitimate value
of convex functions, and we use relations such as $\exp(-\infty) = 0$ whenever they make sense.
The function $\psi$
is locally-Lipschitz and hence differentiable almost everywhere in  the interior of the set
$$ \{ \psi < + \infty \} = \{ x \in \RR^n \, ; \, \psi(x) < +\infty \}. $$
In K\"ahler geometry, the map
$$ x \longrightarrow\nabla \psi(x), $$
defined almost-everywhere in $\{ \psi < +\infty \}$, is closely related to the {\it moment map} of a toric K\"ahler manifold, see, e.g.
Abreu~\cite{abreu} or Gromov~\cite{gromov}. When the
function
$\psi$ is finite and smooth, the set
\begin{equation}  \nabla \psi(\RR^n) = \left \{ \nabla \psi(x) ; x \in \RR^n \right \} \label{eq_1708} \end{equation}
is necessarily convex. In certain cases of a interest the convex set (\ref{eq_1708}) is in fact a polytope, which is referred to as the {\it moment polytope}; a central role is played by a family of polytopes known as \emph{Delzant polytopes}  which carry a particular geometric structure.

\medskip
In this article, we  consider  convex functions $\psi: \RR^n \rightarrow \RR \cup \{ + \infty \}$  that satisfy the integrability condition $0 < \int \exp(-\psi) < \infty$. This
condition, for a convex function $\psi$, is equivalent to the following two requirements:
\begin{enumerate}
\item[(i)] The convex set $\{ \psi < +\infty \}$ is not contained in a hyperplane; and
\item[(ii)] $\displaystyle \lim_{x \rightarrow \infty} \psi(x) = +\infty. $
\end{enumerate}
We associate with such $\psi$ the finite (log-concave) measure $\mu_{\psi}$  on $\RR^n$ whose density is  $\exp(-\psi)$.

\begin{defi}
Given a convex function $\psi: \RR^n \rightarrow \RR \cup \{ + \infty \}$
with $0 < \int \exp(-\psi) < \infty$, we define its moment measure $\mu$ to
be the Borel measure on $\R^n$ which is the push-forward of $\mu_\psi$ under $\nabla\psi$.
This means that
\begin{equation}\label{eq:defimage}
\int_{\RR^n} b(y) \, d\mu(y) = \int_{\RR^n} b(\nabla \psi(x)) \, e^{-\psi(x)}\, dx
\end{equation}
for every Borel function $b$ such that
$b \in L^1(\mu)$ or $b$ is nonnegative.
\end{defi}

Note that translating $\psi(\cdot)$ to $\psi(\cdot - v_0)$, with $v_0\in \R^n$, leaves the moment measure unchanged. Adding a constant $\lambda$ to $\psi$ multiplies $\mu$ by $e^{-\lambda}$. It is therefore costless to impose that $\mu$ and $\mu_\psi$ are probability measures, rather than dealing with finite, non-zero measures. A classical  example from  complex geometry is given by the function
$$ \psi(x) = (n+1) \log \left[ \sum_{i=0}^n \exp \left( \frac{x \cdot v_i}{n+1} \right) \right] \quad \quad \quad \quad (x \in \RR^n) $$
where $v_0,\ldots,v_n \in \RR^n$ are $n+1$ vectors that add to zero and span $\RR^n$, and $x \cdot y$ stands for the standard
scalar product of $x,y \in \RR^n$.
A computation (see, e.g. \cite{gromov} or \cite{k_poincare}) shows
that the moment measure of $\psi$ is proportional to the uniform probability measure on the simplex
whose vertices are $v_0,\ldots,v_n$. The case of uniform measures on
Delzant polytopes  has also been studied by complex geometers in connection with the structure of toric varieties.  In general, it is not very easy to describe the convex function $\psi$ whose moment
measure is a given Borel measure $\mu$ in $\RR^n$.
For instance, see  Bunch and Donaldson \cite{BD} or Doran, Headrick, Herzog, Kantor and Wiseman
\cite{physicists} for a numerical approximation of $\psi$
in the case where $\mu$ is the uniform probability measure on a hexagon centered at the origin in $\RR^2$.

\medskip
One would like  to understand which measures $\mu$ on $\RR^n$ are moment measures  of a convex function.
The case where $\mu$ is supported on a convex body of $\RR^n$, and has a smooth density
bounded from below and from above by positive constants on this convex body, was successfully studied by Wang and Zhu \cite{WZ}, Donaldson \cite{donaldson}
and
Berman and Berndtsson \cite{BB}. Here, a convex body means a non-empty, bounded, open, convex set.
Our aim is to complete the description of moment measures, by giving necessary and sufficient conditions. As we shall see, one can go way beyond the case of smooth functions on convex bodies.

\medskip
A preliminary, somehow converse, question  is to know whether one can recover the
convex function $\psi$ from its moment measure. The answer, in general, is negative.  For instance, given a convex body $C \subset \RR^n$
and a vector $y \in \R^n$, consider the convex function
\begin{equation}  \psi(x) = \left \{ \begin{array}{cc} x \cdot y & x \in C \\ +\infty & x \not \in C \end{array} \right. \label{eq_1132}
\end{equation}
Then the moment measure of $\psi$ is just a multiple of $\delta_{y}$, the Dirac measure at the point $y$. It is therefore impossible to recover any relevant information on $\psi$ and $C$
from the moment measure. The obstacle seems to be the {\it discontinuity} of the convex function $\psi$.
We shall see below that convex functions from $\RR^n$ to $\RR \cup \{ + \infty \}$ that are {\it continuous} are
much more well-behaved. Of course, when we refer to continuity at a point where the function is infinite, we mean that the limit at this point is $+\infty$.  We shall see that a property weaker than continuity is in fact sufficient for our purposes. This property deserves its own terminology, not only for the writing convenience,  but also because it will prove to be  natural in the present context.

\begin{defi}We say that a convex function $\psi: \RR^n \rightarrow \RR \cup \{ + \infty \}$ is \emph{essentially-continuous}
if $\psi$ is lower semi-continuous and if the set of points where $\psi$
is discontinuous has zero $\cH^{n-1}$-measure. Here,  $\cH^{n-1}$ is the $(n-1)$-dimensional
Hausdorff measure. \label{def_1055}
\end{defi}

Before going on, we need to make a  few comments about Definition \ref{def_1055}.  First, in dimension one,
essential-continuity is equivalent to continuity. Next, note that  a convex function from $\RR^n$
to $\RR \cup \{ + \infty \}$ is automatically continuous outside
$$ \partial \{ \psi < + \infty \}. $$
Definition \ref{def_1055} is thus concerned only with the boundary behavior of the function
$\psi$ near the set  $\partial \{ \psi < +\infty \}$. In particular, any finite convex function $\psi: \RR^n \rightarrow \RR$ is essentially-continuous. The requirement that $\psi$ is lower semi-continuous
is actually not very drastic, and has no geometric consequences.
It only amounts to the convenient fact that the epigraph of $\psi$ is a closed set, while the second part of Definition~\ref{def_1055} puts severe restrictions on supporting hyperplanes of this convex set. Note that the lower semi-continuity of $\psi$ ensures continuity at points where $\psi$ is $+\infty$. Thus a convex function $\psi$ is essentially-continuous if and only if it is lower semi-continuous  and  if
$$\cH^{n-1}\big( \{ x \in \partial \{\psi < +\infty \} \; ; \ \psi(x)<+\infty\}\big) =0, $$
or equivalently,
$$ \{ \psi < +\infty \} = A \cup B $$
where $A$ is an open convex set and $B\subseteq \partial A$ is such that $\cH^{n-1}(B) = 0$.

Our first step is to establish some necessary conditions that are satisfied by
the moment measure of any essentially-continuous, convex function.

\begin{prop} Let $\psi: \RR^n \rightarrow \RR \cup \{+\infty \}$ be an essentially-continuous,
convex function with $0 < \int_{\RR^n} \exp(-\psi) < +\infty$. Then the moment
measure of $\psi$ is not supported in a hyperplane, and its barycenter lies at the origin.
\label{prop_1107}
\end{prop}

Proposition~\ref{prop_1107} will be proven in Section \ref{sec2}, in which we collect
a few simple properties of convex functions, log-concave densities, and moment measures.
It turns out that these \emph{necessary} conditions
are also \emph{sufficient}. Our main result below  is indeed that there is a bijection
between essentially-continuous, convex functions modulo  translations,
and finite measures on $\RR^n$ satisfying the conclusion of Proposition~\ref{prop_1107}.

\begin{theo} Let $\mu$ be a  Borel  measure on $\RR^n$ such that
\begin{enumerate}
\item[(i)] $\displaystyle 0 < \mu(\RR^n) < +\infty$.
\item[(ii)] The measure $\mu$ is not supported in a lower-dimensional subspace.
\item[(iii)] The barycenter of $\mu$ lies at the origin (in particular, $\mu$ has finite first moments).
\end{enumerate}
Then there exists a convex function $\psi: \RR^n \rightarrow \RR \cup \{ + \infty \}$, essentially-continuous, such that $\mu$ is the moment measure of $\psi$.
Moreover, such $\psi$ is  uniquely determined up to translation.
\label{thm_1459}
\end{theo}

The picture now fits nicely with constructions from optimal transportation theory. Indeed, once the existence of $\psi$ for a given $\mu$ has been established, as in Theorem~\ref{thm_1459}, then we see that $\nabla\psi$ is the (unique) convex gradient map pushing forward $e^{-\psi(x)} \, dx$ to $\mu$. This map is the \emph{quadratic-optimal map}, also known as the \emph{Brenier map}, between these measures (see for instance~\cite{mccann95}). In particular, if $\mu$ is absolutely-continuous with respect to the Lebesgue measure, and if we write $d\mu(x) = g(x) \, dx$ with $g\in L^1$, then, as established by McCann~\cite{mccann}, the Monge-Amp\`ere equation
\begin{equation}\label{ma}
e^{-\psi(x)} = g(\nabla\psi(x)) \, \det D^2\psi(x)
\end{equation}
is verified in a measure-theoretic sense. Namely, this change-of-variables equation is verified $e^{-\psi(x)} \, dx$-almost everywhere, provided that the Hessian $D^2\psi$ is understood (almost everywhere) in the sense of Aleksandrov, i.e. as the second order term in the Taylor expansion, or as the derivative of the sub-gradient map, or as the density of the absolutely-continuous part of the distributional Hessian; some extra explanations will be given later on.
Of course, under further assumptions on $g$, as in Berman and Berndtsson~\cite{BB}, one can call upon the regularity theory for Monge-Amp\`ere equations and conclude that we have a classical solution to~\eqref{ma}.

\medskip
In dimension one, it is not very difficult to express the function $\psi$ from Theorem~\ref{thm_1459} in terms of $\mu$.
For instance, when the support of $\mu$ is connected, the convex function $\psi$ is differentiable and it satisfies
the equation
$$\left( \psi^{-1} \right)^{\prime} \left \{ -\log \left( \int_y^{\infty} t d \mu(t) \right) \right\} = \frac{1}{y}.
$$ where $\psi^{-1}$ is any local inverse of $\psi$, that is, $\psi^{-1}$ may stand for any function defined on an interval $I \subset \RR$
with $\psi \circ \psi^{-1} = Id$. There are easy cases in higher dimension as well.
For instance, the moment measure of the convex function $|x|^2/2 = (x \cdot x) / 2$ is proportional
to the standard Gaussian measure in $\RR^n$. The uniform measure on the sphere $S^{n-1}=\{x\in \R^n\; ; \ |x|=1\}$ is proportional to the moment measure of the convex function $\psi(x)=|x|$. The uniform probability measure on the cube $[-1,1]^n$ is proportional to the moment measure of the convex
function
\begin{equation}
 \psi(x) = \sum_{i=1}^n 2 \log \cosh \left( \frac{x_i}{2} \right) \quad \quad \quad \quad \text{for} \ x = (x_1,\ldots,x_n) \in \RR^n.  \label{eq_1455_}
 \end{equation}
By linear invariance, we may express the uniform probability measure on a centered parallelepiped in $\RR^n$ as the moment measure
of
$$ x \to \psi(T(x)) + C $$
for some linear map $T: \RR^n \rightarrow \RR^n$ and $C \in \RR$, where $\psi$ is as in (\ref{eq_1455_}).

\medskip
Our proof of Theorem~\ref{thm_1459} follows the variational approach promoted by Berman and Berndtsson~\cite{BB},
which is a distant cousin of the original approach by Minkowski (see Schneider \cite[Section 7.1]{schneider} and references therein). Our treatment is however different than the one in~\cite{BB}.
We shall see that Theorem~\ref{thm_1459} is intimately related to a variant of the Pr\'ekopa
inequality, which is described in Section~\ref{sec3} below. The proof of Theorem~\ref{thm_1459} (existence and unicity)
is completed in Section~\ref{sec4}. In Section \ref{sec5} we discuss potential generalizations of Theorem~\ref{thm_1459} and
its relations to the Minkowski problem and to the logarithmic Minkowski problem.

\medskip A final warning: As the reader will notice, and as is apparent from the number of lemmas, some parts of the paper may seem a bit technical, hopefully not too much. This is partly due to the fact that we need to pay close attention to the domain of the convex functions and to the support of the measures.  Actually, these issues are not purely technical: they encode part of the geometry of the problem and they ensure clean statements. We hope that the reader will be convinced that they should not be overlooked.

\medskip {\it Acknowledgement.} We would like to thank Bo Berndtsson and Yanir Rubinstein
for their enlightening explanations concerning the theory of K\"ahler-Einstein equations in toric varieties and for their interest in this work.
The second named author would also like to thank the Fondation des Sciences Math\'ematiques de Paris (FSMP)
for funding his visit to Paris during which
most of this work was done, and to the European Research Council (ERC) for supporting his research.


\section{Basic Properties of moment measures}
\label{sec2}

In this section, we establish several useful properties of the moment measures $\mu$  and of the log-concave measures $\mu_\psi$.
The main use of {\it essential continuity}
will be through the following lemma.

\begin{lemma} Let $\psi: \RR^n \rightarrow \RR \cup \{ + \infty \}$
be an essentially-continuous, convex function. Fix a vector $0 \neq \theta \in \RR^n$
and let $H = \theta^{\perp} \subset \RR^n$ be the hyperplane orthogonal to $\theta$. Then, for $\cH^{n-1}$-almost every $y \in H$, the function
\begin{eqnarray}
\R & \to & \R\cup\{+\infty\} \nonumber \\
t &\to & \psi(y + t \theta) \label{eq_1722_}
\end{eqnarray}
is continuous on $\RR$, and  locally-Lipschitz in the interior of the interval in which it is finite. \label{ess}
\end{lemma}

\begin{proof} The set of discontinuity points of $\psi$ has a zero $\cH^{n-1}$-measure,
and the same is true for  its orthogonal projection onto the hyperplane $H$. Therefore the function in (\ref{eq_1722_})
is continuous in $t$, for $\cH^{n-1}$-almost any $y \in H$. The function in (\ref{eq_1722_})  is convex, hence it is locally-Lipschitz
in the interior of the interval in which it is finite. \end{proof}

 A {\it log-concave function}
is a function of the form $\exp(-\psi)$ where $\psi: \RR^n \rightarrow \RR \cup \{ + \infty \}$ is convex.
A log-concave function is degenerate if it vanishes almost-everywhere in $\RR^n$.
It is well-known (see, e.g., \cite[Lemma 2.1]{K_psi}) that a non-degenerate, log-concave function
$\exp(-\psi)$ is integrable on $\RR^n$
if and only if
\begin{equation}  \liminf_{|x| \rightarrow \infty} \frac{\psi(x)}{|x|} > 0, \label{eq_1726} \end{equation}
where $| \cdot |$ is the standard Euclidean norm in $\RR^n$. Equivalently, $\exp(-\psi)$ is integrable on $\RR^n$ if and only if
$$  \lim_{x \rightarrow \infty} \psi(x) = +\infty.$$
Any log-concave function is differentiable almost everywhere in $\RR^n$. The next Lemma establishes an integrability result for log-concave gradients, that will ensure finite first moments for the moment measure of a convex function. The second part of the lemma gives more information on the barycenter under the assumption that the convex function is essentially-continuous. Throughout the paper, $\partial^i$ and $\partial^{i j}$ will stand for the partial derivatives (of first and second order, respectively) in the canonical basis of $\R^n$.

\begin{lemma} Let $\rho: \RR^n \rightarrow [0, +\infty)$ be an integrable, log-concave function. Then,
\begin{equation}  \int_{\RR^n} |\nabla \rho| < +\infty. \label{eq_1756}
\end{equation}
Furthermore, in the case where $\rho = \exp(-\psi)$ with $\psi$ essentially-continuous, we have
$$  \int_{\RR^n} \partial^i \rho = 0 \quad \quad \quad \quad (i=1,\ldots,n). $$
\label{lem_1439}
\end{lemma}

In order to appreciate this simple
observation, let us remark that there exist integrable and smooth log-concave functions  $\rho: \RR^n \rightarrow [0, +\infty)$ such that $\nabla \rho\notin L^{1+\eps}(\R^n)$ for any $\eps > 0$.

\begin{proof}
To prove (\ref{eq_1756}), it suffices to show that
\begin{equation}
 \int_{\RR^n} |\partial^i \rho| < +\infty. \label{eq_1634_}
 \end{equation}
 for any $i=1,\ldots,n$. Without loss of generality fix $i=n$ and write, for $x \in \RR^n$,
 $$ x = (y,t) \quad \quad (y \in \RR^{n-1}, t =x_n\in \RR). $$
For any fixed $y \in \RR^{n-1}$, the function $t \to \rho(y,t)$ is log-concave, non-negative, locally-Lipschitz in the interior of the interval
in which it is positive, and tends to $0$ at $\pm \infty$ by~\eqref{eq_1726}; in particular this function is non-decreasing on a half-line, and non-increasing on the complement. So we have,
$$ \int_{-\infty}^{\infty} \left | \frac{\partial \rho(y,t)}{\partial t} \right| \, dt \leq 2 \sup_{t \in \RR} \rho(y,t). $$
By Fubini,
\begin{align*}
 \int_{\RR^n} |\partial^i \rho(x)| \, dx   = \int_{\RR^{n-1}} \int_{-\infty}^{\infty} \left | \frac{\partial \rho(y,t)}{\partial t} \right| \, dt \, dy \leq 2 \int_{\RR^{n-1}} \exp \left(
-\inf_{t \in \RR} \psi(y,t)
\right) \, dy.
\end{align*}
The function $\psi_1(y) = \inf_{t \in \RR} \psi(y,t)$ is convex. Since $\psi$ satisfies (\ref{eq_1726}) then $\psi_1$ satisfies
$$  \liminf_{|y| \rightarrow \infty} \frac{\psi_1(y)}{|y|} > 0. $$
Hence $\exp(-\psi_1)$ is integrable, and  (\ref{eq_1634_}) is proven. In order to prove
the ``Furthermore'' part, we use Lemma \ref{ess}. For almost any $y \in \RR^{n-1}$, the function $t \to \rho(y,t)$
is continuous, vanishes at infinity, and it is locally-Lipschitz in the interior of its support.
Therefore, for almost any $y \in \RR^{n-1}$,
$$
\int_{-\infty}^{\infty} \frac{\partial \rho(y,t)}{\partial t}  \, dt = 0. $$
Thanks to (\ref{eq_1634_}) we may use Fubini's theorem, and conclude that
\begin{equation*} \int_{\RR^n} \partial^i \rho = \int_{\RR^{n-1}} \left( \int_{-\infty}^{\infty} \frac{\partial  \rho(y,t)}{\partial t}  \, dt \right) \, dy = 0. \tag*{\qedhere}
\end{equation*}

\end{proof}

Next we establish an integration by parts inequality.

\begin{lemma} Suppose that $\psi: \RR^n \rightarrow \RR \cup \{+\infty \}$ is a convex function with $0 < \int
\exp(-\psi) < +\infty$. As before, we write $\mu_{\psi}$ for the measure with density $\exp(-\psi)$. Then the function $x \cdot \nabla \psi(x)$
is $\mu_{\psi}$-integrable. If $\psi$ is furthermore essentially-continuous, then also
$$ \int_{\RR^n} \left[ x \cdot \nabla \psi(x) \right] d \mu_{\psi}(x) \leq n \int_{\RR^n} e^{-\psi}. $$
\label{lem_1126}
\end{lemma}

\begin{proof} The convex set $\{ \psi < +\infty \}$ has a non-empty interior, as $\exp(-\psi)$
is a non-degenerate, log-concave function. Pick a point $x_0$ in the interior of $\{ \psi < +\infty \}$.
From the convexity of $\psi$, for any point $x$ in which $\psi$ is differentiable,
\begin{align*} \nabla \psi(x) \cdot (x - x_0) \geq \psi(x) - \psi(x_0)
 \geq -C_{x_0} \end{align*}
with $C_{x_0} = \psi(x_0) - \inf \psi$. Note that $C_{x_0}$ is finite according to (\ref{eq_1726}).
We thus see that the function $x \to \nabla \psi(x) \cdot (x-x_0) $ is bounded from below.
In order to bound its integral from above,
we use
\begin{equation}  \int_{\RR^n} \left[ \nabla \psi(x) \cdot (x -x_0) \right] d \mu_{\psi}(x) \leq \sup_{K \subset \RR^n}
\int_K \left[ \nabla \psi(x) \cdot (x -x_0) \right] d \mu_{\psi}(x) \label{eq_1414_} \end{equation}
where the supremum runs over all compact sets $K$ contained in the interior of $\{ \psi < + \infty \}$.
Since $\{ \psi < +\infty \}$ is convex, it suffices to restrict attention in (\ref{eq_1414_}) to convex, compact sets $K$,
contained in the interior of $\{ \psi < + \infty \}$, that include $x_0$
in their interior. We may even enlarge $K$ a little bit and assume that it has a smooth boundary.
For such $K$, the function $\exp(-\psi)$ is Lipschitz in $K$, and we may use the divergence theorem:
\begin{align} \label{eq_1131} \int_{K} & \left[ \nabla \psi(x) \cdot (x - x_0) \right] e^{-\psi(x)} \, dx  =  \int_{K} \left[ n e^{-\psi} - div( (x - x_0) e^{-\psi} )\right] \, dx \\
& = n \int_{K} e^{-\psi} - \int_{\partial K} e^{-\psi} \left[ (x -x_0) \cdot \nu_x \right] \, dx \leq n \int_{K} e^{-\psi} \leq n \int_{\RR^n} e^{-\psi}
\nonumber
\end{align}
where $\nu_x$ is the outer unit normal to $\partial K$ at the point $x \in \partial K$, and where we used the fact that for any $x \in \partial K$,
$$  x_0 \cdot \nu_x \leq x \cdot \nu_x = \sup_{y \in K} y \cdot \nu_x $$
as $x_0 \in K$ and $K$ is convex. From (\ref{eq_1414_}) and (\ref{eq_1131}),
$$ \int_{\RR^n} \left[ \nabla \psi(x) \cdot (x -x_0) \right] d \mu_{\psi}(x)  \leq n \int_{\RR^n} e^{-\psi}. $$
We have thus shown that the function $x \to \nabla \psi(x) \cdot (x -x_0)$ is $\mu_{\psi}$-integrable,
and the integral is at most $n \int \exp(-\psi)$.
Lemma \ref{lem_1439} implies that $x \to \nabla \psi(x) \cdot x_0$ is $\mu_{\psi}$-integrable,
and that in the essentially-continuous case, we also have $\int_{\RR^n} [\nabla \psi(x) \cdot x_0] d \mu_{\psi}(x) = 0$. The conclusion of the lemma follows.
\end{proof}

The inequality of Lemma \ref{lem_1126} is in fact an equality, but we will neither use nor prove this
equality in this paper.

Recall that the support of a measure $\mu$ in $\RR^n$ is the closed
 set $Supp(\mu)$ that consists  of all points $x \in \RR^n$ with the following property: $\mu(U) > 0$ for any open set $U$ containing $x$.

\begin{lemma}\label{lemma:supp} Let $\psi: \RR^n \rightarrow \RR \cup \{ + \infty \}$ be an essentially-continuous, convex function
with $0 < \int \exp(-\psi) < +\infty$. Let $\mu$ be its moment measure. Then $Supp(\mu)$ is not contained in any
subspace $E \subsetneq \RR^n$.  \label{lem_1631}
\end{lemma}

\begin{proof} Denote $\rho = \exp(-\psi)$. Assume by contradiction that $Supp(\mu) \subseteq \theta^{\perp}$ for a vector $\theta \in \R^n$, $|\theta|=1$.
Without loss of generality, assume that $\theta = e_n$, where
$e_n = (0,\ldots,0,1)$.
For $x \in \RR^n$, we write $ x = (y,t) \ \ (y \in \RR^{n-1}, t=x_n \in \RR) $.
According to our assumption, $0=\int |z_n| \, d\mu(z)=  \int \big|\frac{\partial \psi}{\partial x_n} \big|e^{-\psi(x)}\, dx$, so for almost every $(y,t) \in \RR^n$,
$$ \frac{\partial \rho(y,t)}{\partial t} = 0. $$
According to Lemma \ref{ess}, for almost any $y \in \RR^{n-1}$, the function
$t \to \rho(y, t) $
is continuous in $\RR$ and locally-Lipschitz in the interior of the interval $\{ t ; \rho(y,t) > 0 \}$.
 Therefore, for almost any $y \in \RR^{n-1}$,
the log-concave function $ t \to \rho(y,t) $ is constant in $\RR$ and so $\int  \rho(y,t)\, dt \in \{0, +\infty\}$.
By Fubini, the function $\rho$ cannot have a finite, non-zero integral --
in contradiction.
\end{proof}

Now we have all of the ingredients required in order to
establish the necessary conditions satisfied by
moment measures.

\begin{proof}[Proof of Proposition \ref{prop_1107}]
From Lemma \ref{lem_1439} and the definition
of the moment measure, the barycenter
of $\mu$ lies at the origin.
Lemma \ref{lem_1631} tells us that
the support of $\mu$ cannot be contained in a hyperplane through the origin.
Since its barycenter is at the origin, then $Supp(\mu)$ cannot be contained either in a hyperplane
that does not pass through the origin.
\end{proof}

The end of this section is devoted to some connections between the Legendre transform,  gradient maps and moment measures, that are at the heart of our study. The {\it subgradient} of the convex function $\psi: \RR^n \rightarrow \RR \cup \{ +\infty \}$ at the
point $x_0 \in \{ \psi < +\infty \}$ is
$$ \partial \psi(x_0) = \left \{ y \in \RR^n \, ; \, \forall x \in \RR^n, \ \psi(x) \geq \psi(x_0) + y \cdot (x - x_0)  \right \}. $$
See, e.g., Rockafellar \cite[Section 23]{rockafellar} for a thorough discussion of subgradients of convex functions.
For completeness, we write $\partial \psi(x_0) = \emptyset$ when $\psi(x_0) = +\infty$. The closed, convex set $\partial \psi(x_0)$
is non-empty whenever $\psi$ is finite in a neighborhood of $x_0$. It equals $\{ \nabla \psi(x_0) \}$ whenever $\psi$ is differentiable at $x_0$.

Let $\psi: \RR^n \rightarrow \RR \cup \{ + \infty \}$ be a function, convex or not,
which is not identically $+\infty$.
Its Legendre transform is defined as
$$ \psi^*(y) = \sup_{x \in  \RR^n} \left[ x \cdot y - \psi(x) \right] \quad \quad \quad \quad (y \in \RR^n), $$
where the supremum runs over all $x \in \RR^n$ with $\psi(x) < +\infty$.
The function  $\psi^*:\RR^n \rightarrow \RR \cup \{ + \infty \}$ is always convex and lower semi-continuous.
When $\psi$ is convex and  differentiable at the point $x$, we have
\begin{equation}  \psi^*(\nabla \psi(x)) = x \cdot \nabla \psi(x) - \psi(x). \label{eq_2156} \end{equation}
When $\psi$ is convex and lower semi-continuous, it is true that $(\psi^*)^* = \psi$.
As can be seen from (\ref{eq_1726}),  for a non-degenerate, log-concave function $e^{-\psi}$ in $\RR^n$, we have
\begin{equation}
\int_{\RR^n} e^{-\psi} <+\infty\  \Longleftrightarrow
\ 0 \ \textrm{ belongs to the interior of } \ \{ \psi^* < + \infty \}.
\label{eq:integrability}
\end{equation}

Note that formally, if $\psi$ lives on $X=\R^n$, then its Legendre transform $\psi^\ast$ lives on $X^\ast=\R^n$, which is also the space where the moment measure of $\psi$ lives. So let us keep in mind this fact that $\varphi=\psi^\ast$ and the moment measure of $\psi$ live on the same space of variables, and let us notice the following integrability property.

\begin{prop} Suppose that $\psi: \RR^n \rightarrow \RR \cup \{+\infty \}$ is a convex function with $0 < \int
\exp(-\psi) < +\infty$. Set $\vphi = \psi^*$ and denote by $\mu$ the moment measure of $\psi$. Then $\vphi$ is $\mu$-integrable.
\label{prop_1114}
\end{prop}

\begin{proof} Recall that $\mu_{\psi}$ is the measure with density $\exp(-\psi)$. Equality (\ref{eq_2156}) is valid $\mu_{\psi}$-almost everywhere.
Thus, in order to conclude the lemma, it suffices to show that
\begin{equation}
\psi \in L^1(\mu_{\psi})\quad \quad \text{and} \quad \quad  x \cdot \nabla \psi(x) \in L^1(\mu_{\psi}). \label{eq_1736}
 \end{equation}
 The second assertion in (\ref{eq_1736}) holds in view of Lemma \ref{lem_1126}. For the first assertion,
 we use the fact that the function $\psi$ is bounded from below as it satisfies (\ref{eq_1726}).
With the help of the inequality $x e^{-x} \leq 2 e^{-x/2}$, valid for all $x \in \RR$,
we deduce that
\begin{equation}
 -\infty < \int_{\RR^n} \psi \, d \mu_{\psi} = \int_{\RR^n} \psi e^{-\psi} \leq 2 \int_{\RR^n} e^{-\psi/2} < + \infty
 \label{eq_1733}
 \end{equation}
where the upper bound follows from (\ref{eq_1726}). Thus $\psi \in L^1(\mu_{\psi})$.
\end{proof}

Let us also mention that, under the assumptions of Lemma \ref{lem_1631}, one may actually reach the conclusion that
\begin{equation}  conv(Supp(\mu)) = \overline{ \{\psi^* < +\infty \} },
\label{eq_1515}
\end{equation}
where $\overline{A}$ denotes the closure of $A$, and where for $A \subset \RR^n$ we write  $conv(A)$ for its convex hull.
This stronger conclusion will not be needed here.

What will be used below are the following  elementary observations regarding the convex hull of the support of a Borel measure.
Suppose that $\mu$ is a finite Borel measure on $\RR^n$ whose support is not contained in a lower-dimensional subspace.
First, note that the barycenter of $\mu$
is always contained in the interior of $conv(Supp(\mu))$.
Next, observe that  if a convex function $\varphi:\R^n\to \R\cup\{+\infty\}$ is $\mu$-integrable, then $\{\varphi <+\infty\}$ contains the interior of $conv(Supp(\mu))$. Indeed, otherwise $\varphi$ would be infinite in a half-space that has positive $\mu$-measure.
A quantitative version of this fact will be given in Lemma~\ref{lem_2211} below.


\section{A version of Pr\'ekopa's inequality}
\label{sec3}

In this section, we establish a \emph{subgradient} or \emph{above-tangent} version of Pr\'ekopa's inequality which will be used in the proof of the
uniqueness of $\psi$ in Theorem~\ref{thm_1459} and which also has  strong connections with the variational problem used to prove existence. But the statement is  of independent interest. Its proof relies on monotone transport (Brenier map). At the end of the section, we have included a converse statement,
that demonstrates the central role played by essential-continuity in the present context, and that is also necessary for the proof of existence of $\psi$.

\medskip
Pr\'ekopa's theorem (which is a particular case of the Pr\'ekopa-Leindler inequality when the potentials are convex) states that for a given convex function $\Phi:\R\times \R^n \to \R\cup\{+\infty\}$, the function $\lambda\to -\log \int_{\R^n} e^{-\Phi(\lambda,x)}\, dx$ is convex on $\R$. We refer to e.g.~\cite{mccann, maurey, cordero_klartag} for background and recent developments, including connections to complex analysis.

\medskip
It is possible to rewrite Pr\'ekopa's inequality using the fact that the Legendre transform linearizes infimal convolution. More precisely,
let $u_0, u_1: \RR^n \rightarrow \RR \cup \{ + \infty \}$ be any two functions, finite in a neighborhood of the origin.
The Pr\'ekopa inequality applied to the convex functions $u_0^\ast$ and $u_1^\ast$ states  that for any $0 < \lambda < 1$,
\begin{equation}
 \int_{\RR^n} e^{-[(1-\lambda) u_0 + \lambda u_1]^*} \geq \left( \int_{\RR^n} e^{-u_0^*} \right)^{1-\lambda}
\left( \int_{\RR^n} e^{-u_1^*} \right)^{\lambda}. \label{eq_1659} \end{equation}
Indeed, the reader may readily check  that we always have the Pr\'ekopa condition:
$$\forall x,y\in \R^n , \qquad
[(1-\lambda) u_0 + \lambda u_1]^* \big((1-\lambda)x + \lambda y\big)
\; \le\;  (1-\lambda) u_0^\ast(x) + \lambda u_1^\ast(y).$$
In order to deduce (\ref{eq_1659}), set  $\Phi(\lambda, z) := \inf\big\{ (1-\lambda) u_0^\ast(x) + \lambda u_1^\ast(y) \; ; \ z= (1-\lambda)x + \lambda y\big\}$ and with the notation above, apply Pr\'ekopa's theorem to the convex function $\Phi$, that dominates $(\lambda, z)\to[(1-\lambda) u_0 + \lambda u_1]^*(z)$.

\medskip
H\"older's inequality asserts the \emph{concavity} of the functional $u\to -\log \int e^{-u}$.
As it turns out, the convexity of this functional is reversed under the Legendre transform.
Indeed, the previous discussion shows that  Pr\'ekopa's inequality expresses exactly the \emph{convexity} of the functional
$$ \cJ(u) = -\log \int_{\RR^n} e^{-u^*}$$
on the set of all functions $u: \RR^n \rightarrow \RR \cup \{ + \infty \}$  that are finite in a neighborhood of the origin.
In particular, according to~\eqref{eq:integrability}, $\cJ(u) \in \RR \cup \{ + \infty \}$ is well-defined on  the (convex) set  of convex functions $u$ that are finite in a neighborhood of $0$, and $\cJ(u)$ is finite on the subset of such $u$'s that verify $\int e^{-u^\ast} >0$.
And  as far as the convexity of $\cJ$ is concerned, we can indeed deal, without loss of generality, with convex functions $u$ only,  since we always have
$$
[(1-\lambda) u_0 + \lambda u_1]^*  \le [(1-\lambda) u_0^{\ast \ast} + \lambda u_1^{\ast \ast}]^* .$$

The following theorem states that (minus) the normalized
moment measure  can be interpreted as a {\it differential} or {\it first variation} or {\it tangent} to the convex functional~$\cJ$.

\begin{theo}[Pr\'ekopa's theorem revisited]
Let $e^{-\psi_0}$ and $e^{-\psi_1}$ be two log-concave functions on $\R^n$ with  $0<\int e^{-\psi_0} <+\infty$
and such that $\psi_1$ is not identically $+\infty$. Assume that $\psi_0$ is essentially-continuous.
Then, writing $\varphi_0=\psi_0^\ast$ and $\varphi_1=\psi_1^\ast$, we have
\begin{equation} \log \int_{\RR^n} e^{-\psi_0} -\log\int_{\RR^n} e^{-\psi_1} \ge  \int_{\RR^n} (\varphi_0- \varphi_1)\, d\overline\mu, \label{eq_1050} \end{equation}
where $\overline\mu$ is the probability measure on $\RR^n$ that is proportional to the moment measure of $\psi_0$, i.e. $\overline\mu$ is the push-forward of the probability density $ \frac{e^{-\psi_0(x)}}{\int e^{-\psi_0}}\, dx$ under the map $\nabla \psi_0$.
\label{thm_prekopa}
\end{theo}

Regarding the interpretation of (\ref{eq_1050}):
Since $\psi_1$ is not identically $+\infty$
then the function  $\vphi_1$ is bounded from below by an affine function, which is $\overline\mu$-integrable.
Proposition~\ref{prop_1114} asserts that $\vphi_0$ is $\overline\mu$-integrable.
Consequently, the right-hand side of (\ref{eq_1050}) is in $\RR \cup \{-\infty\}$.
Since $e^{-\psi_0}$ has a finite, non-zero integral, then the left-hand
side of (\ref{eq_1050}) is in $\RR \cup \{ \pm \infty \}$, and the inequality always makes sense.

\medskip
The proof of Theorem \ref{thm_prekopa} uses monotone transportation of measure,
which has been for a long time a standard approach to Brunn-Minkowski type inequalities, such as the Pr\'ekopa
inequality (the modern story starts with McCann's work~\cite{mccann}). More specifically, we will exploit here the  differential point view  used in~\cite{cordero} to prove logarithmic-Sobolev inequalities.

\medskip
Let $f$ be a convex function in $\RR^n$. Recall that $f$ is differentiable $\cH^n$-almost everywhere in $\{ f< +\infty \}$.
According to the Aleksandrov theorem (see, e.g., Evans and Gariepy \cite[Section 6.4]{EG}), the convex function $f$ admits a Taylor
expansion of order two
$$f(x_0+h) = f(x) + \nabla f(x_0) \cdot h + \frac12 H_{x_0} h \cdot h + o(|h|^2)$$
at $\cH^n$-almost every $x_0$  in the set $\{ f< + \infty \}$.
Here $H_{x_0}$ is a positive semi-definite $n \times n$ matrix.
This second order term coincides also ($\cH^n$-almost-everywhere) with the derivative of the set valued map $\partial f$ which is defined, $\cH^n$-almost everywhere, at points $x_0$ where $\nabla f(x_0)$ exists, by the property that
$$\lim_{h\to 0}\sup_{y \in \partial f(x_0+h)} \frac{\left| y -  \nabla f(x_0) - H_{x_0} h \right|}{|h|} =0.
$$
We may therefore speak
of the Hessian matrix $D^2 f(x)=H_x$, defined
$\cH^n$-almost everywhere in $\{ f< +\infty \}$.
Whenever we mention second derivatives of convex functions,
we refer to this ``second derivative in the sense of Aleksandrov''.

\medskip
We start with a technical one-dimensional lemma that will be used to justify the integration by parts.

\begin{lemma} Let $\rho = \exp(-\psi)$ be an integrable, continuous, log-concave function on $\RR$.
Set $d \mu_{\psi} = \rho(x) \, dx$. Let $f: \RR \rightarrow \RR$ be a Lipschitz, convex function. Then,
$$ \int_{-\infty}^{\infty} \left( f^{\prime \prime} - f^{\prime} \psi^{\prime} \right) d \mu_{\psi} \leq 0. $$
\label{lem_1D}
\end{lemma}
Of course, under stronger smoothness assumptions, the inequality above is an equality. But the inequality, which holds without further assumptions, is sufficient for our purposes.

\begin{proof} Denote $V = f^{\prime}$ the right-derivative of the convex function $f$. It is a bounded, non-decreasing, right-continuous  function on $\RR$ (and continuous except maybe on a countable number of points). It has a derivative $V'$ almost everywhere which coincides almost everywhere with $f''$, the Aleksandrov second derivative. (Here ``Aleksandrov reduces to Lebesgue''). Since $\rho$ is non-negative and bounded,
then we may bound from below the Lebesgue-Stieltjes integral: For any $[a,b] \subset \RR$,
$$ \int_a^b \rho \, d V \geq \int_a^b \rho(x) V^{\prime}(x) \,  dx.  $$
Indeed,  $V^{\prime}(x) \,dx$ is the absolutely continuous part of the Lebesgue-Stieltjes measure $d V$, see e.g., Stein and Shakarchi \cite[Section 6.3.3]{SK}.
So for any interval $[a,b] \subset \RR$, we have
\begin{equation}  \int_{a}^{b} f^{\prime \prime} e^{-\psi} = \int_{a}^{b}
\rho(x) V^{\prime}(x) \,  dx \leq \int_a^b \rho\, d V = -\int_a^b V \, d \rho + \rho(b)V(b-0) - \rho(a) V(a+0),
\label{eq_1117}
\end{equation}
where we are allowed to use the integration by parts formula for the Lebesgue-Stieltjes integral since
$\rho$ is continuous and of bounded variation.
The function $V$ is bounded, and hence  integrable with respect to $d \rho$.
The function $\rho$ is absolutely-continuous and it vanishes at infinity, hence we may take the limit in (\ref{eq_1117}) and conclude that
\begin{equation*} \int_{-\infty}^{\infty} f^{\prime \prime}(x) e^{-\psi(x)} \, dx\leq -\int_{-\infty}^{\infty} V \, d \rho = -\int_{-\infty}^{\infty} f^{\prime}(x) \rho^{\prime}(x) \, dx
= \int_{-\infty}^{\infty} f^{\prime}(x) \psi^{\prime}(x) e^{-\psi(x)} \, dx.
\tag*{\qedhere}
\end{equation*}
\end{proof}

\begin{proof}[Proof of Theorem \ref{thm_prekopa}.]
We may assume that $\int \exp(-\psi_1) > 0$ as otherwise there is nothing to prove.
The convex function $\psi_1$ is bounded from below by an affine function, and therefore
the function $e^{-\psi_1}$ is integrable on compact sets. Let $L \subset \RR^n$ be a large, open ball, centered at the origin, with
$$ \int_L e^{-\psi_1}= \int_{\R^n}  1_L(x)\, e^{-\psi_1(x)}\, dx > 0, $$
where $1_L$ denotes the indicator function of $L$.
Introduce the
Brenier map $S(x)=\nabla g(x)$ between the normalized densities $e^{-\psi_1}\, 1_L$ and $e^{-\psi_0}$,
where  $g$ is a convex function on $\RR^n$. This means that the map $S$ pushes forward the probability measure
on $\RR^n$ whose density is proportional to $e^{-\psi_1}\, 1_L$, to the probability measure
whose density is proportional to $e^{-\psi_0}$.

\medskip Recall that the Brenier map $S = \nabla g$ is uniquely determined almost-everywhere in the support of the measure $e^{-\psi_1(x)} 1_L(x) \,dx$
(see~\cite{mccann95}), but we still have the freedom to modify $g$ outside the support, as long as the resulting function remains convex. We may therefore stipulate that $g(x) = +\infty$ for $x \not \in L$. Denote $f = g^*$. The convex function $f$ is Lipschitz on the entire $\RR^n$; in fact, its Lipschitz constant is at most the radius of $L$. The map
$$T(x) := \nabla f(x)$$ is the
inverse to $S$, and hence it is the Brenier map between the normalized densities $e^{-\psi_0}$ and $e^{-\psi_1}\, 1_L$.

\medskip By the simple but useful weak-regularity theory of McCann \cite{mccann} (which relies on standard measure-theoretic arguments of Lebesgue type), we have,
for $\mu_{\psi_0}$-almost any $y$,
\begin{equation} \frac{e^{-\psi_0(y)}}{\int_{\RR^n} e^{-\psi_0}}=  \frac{e^{-\psi_1(T(y))}}{\int_{L} e^{-\psi_1}} \det D^2 f(y). \label{eq_2125_} \end{equation}
Here, $\mu_{\psi_0}$ is the measure on $\RR^n$ whose density is $e^{-\psi_0}$,
and $D^2 f(y)$ stands for the Hessian
of the function $f$ in the sense of Aleksandrov. From (\ref{eq_2125_})  we see that
 $\mu_{\psi_0}$-almost everywhere,
 \begin{align*}  \log\int_{L} e^{-\psi_1} - \log \int_{\RR^n} e^{-\psi_0} & =\psi_0(y) - \psi_1(T(y)) + \log\det(D^2f(y) ) \\ & \le \psi_0(y)  - \psi_1(T(y)) + \Delta f(y) - n,
 \end{align*}
 where we used the inequality $\log(s) \le s-1$ for $s\ge 0$ and the fact that $D^2f(y)$ has real nonnegative eigenvalues. Next we use the convexity of $\psi_0$ and $\psi_1$. According to (\ref{eq_2156}), for $\mu_{\psi_0}$-almost any point $y$, we have
$$\psi_0(y) +\varphi_0(\nabla \psi_0(y)) =  \nabla \psi_0(y) \cdot y.$$
On the other hand, for such $y$'s we also have, by the definition of the Legendre transform,
$$\psi_1(T(y)) + \varphi_1 (\nabla\psi_0(y)) \ge  \nabla \psi_0(y) \cdot T(y). $$
Consequently, $\mu_{\psi_0}$-almost everywhere in $\RR^n$,
\begin{equation}
 \log\int_{L} e^{-\psi_1} - \log \int_{\RR^n}  e^{-\psi_0}\le (\varphi_1 - \varphi_0)(\nabla\psi_0(y)) -  \nabla \psi_0(y) \cdot (\nabla f(y) - y) + \Delta f(y) - n. \label{eq_2141} \end{equation}

 The knowledgeable reader has probably identified the term $\Delta f -\nabla \psi_0  \cdot \nabla f $ as the Laplacian associated with the measure $\mu_{\psi_0}$, which should integrate to zero with respect to $\mu_{\psi_0}$. However, we need to be cautious on the justification of the integration by parts since we have not imposed any kind of strong regularity.
Let us  fix $i=1,\ldots,n$, and use $x = (y,t) \ \ (y \in \RR^{n-1}, t = x_i \in \RR)$ as coordinates
in $\RR^n$, where $t$ stands for the $i^{th}$ coordinate and $y$ for all of the rest.  First, we know by Fubini that there exists a set $M\subset \R^{n-1}$ with $\cH^{n-1}(M) = 0$ such that for every $y\in \R^{n-1}\setminus H$,  the Aleksandrov Hessian of $f$ at $(y,t)$, and therefore $\partial^{ii} f (y,t)$, exists for almost every $t\in \R$. Since the function $\psi_0$ is essentially-continuous, we can also assume (Lemma~\ref{ess}) that  the integrable, log-concave function $t \to \exp(-\psi_0(y,t))$ is continuous for every $y\in \R^{n-1}\setminus M$.
Using that $\partial^i f$ is bounded, that $\partial^i \psi_0$ is $\mu_{\psi_0}$-integrable and that $\partial^{ii} f$
is non-negative, we have by Fubini's theorem that:
\begin{align*}
\int_{\RR^n} \left[ \partial^{ii} f -\partial^i \psi_0 \partial^i f   \right] \, d \mu_{\psi_0} = \int_{\RR^{n-1}\setminus M}
\int_{-\infty}^{\infty}  \left[ \frac{\partial^2 f(y,t)}{\partial t^2} - \frac{\partial  \psi_0(y,t)}{\partial t} \frac{\partial f(y,t)}{\partial t} \right] d \mu_{\psi_0, y}(t) \, dy,
\end{align*}
where for $y \in \RR^{n-1} \setminus M$, we write $\mu_{\psi_0, y}$ for the measure on $\RR$ whose density is $t \to e^{-\psi_0(y,t)}$.
We implicitly used the fact  that for any fixed $y\in \R^{n-1}\setminus M$, the derivatives $\partial^i f(y,t)$ and $\partial^{ii} f(y,t)$ coincide, for almost every $t\in \R$, with the first derivative and the second (Aleksandrov) derivative of the convex Lipschitz function $t\to f(y,t)$, respectively.
We may thus apply Lemma \ref{lem_1D}, and conclude that the inner integral above is non-positive for every $y\in \R^n\setminus M$.
Summing over $i=1,\ldots,n$, we have the desired integration by parts inequality
$$ \int_{\RR^n} \left[ \Delta f -\nabla \psi_0  \cdot \nabla f  \right] d \mu_{\psi_0} \leq 0. $$
By combining the last inequality with (\ref{eq_2141}) and Lemma \ref{lem_1126}, we get
$$\log\int_{L} e^{-\psi_1} - \log \int_{\RR^n} e^{-\psi_0}  \le  \frac1{\int e^{-\psi_0}} \int_{\RR^n} (\varphi_1 - \varphi_0)(\nabla\psi_0(y)) \, d \mu_{\psi_0}(y)  = \int_{\RR^n} (\vphi_1 - \vphi_0) d \overline{\mu}.  $$
Since $L$ was an arbitrary large Euclidean ball, the conclusion of the theorem follows.
\end{proof}

The end of this section is devoted to a deeper understanding  of essential-continuity.  It will justify, we hope, the relevance of this notion, and it will be used to show that the function $\psi$ we construct in the proof of Theorem~\ref{thm_1459} will necessarily be essentially-continuous.

 Is the essential-continuity of $\psi_0$ really essential in Theorem \ref{thm_prekopa}? The answer is affirmative,
as the following proposition asserts:

\begin{prop}
Let $\psi:\RR^n \rightarrow \RR \cup \{ + \infty \}$ be  a lower semi-continuous convex function with $\int e^{-\psi} = 1$ and let $\mu$ be the associated moment measure. Denote $\varphi=\psi^\ast$
and assume that

\begin{enumerate}
\item[(*)] For any $\mu$-integrable convex function $\vphi_1: \RR^n \rightarrow \RR \cup \{ + \infty \}$, $\psi_1 = \vphi_1^*$,
\begin{equation}  \log \int_{\RR^n} e^{-\psi} - \log \int_{\RR^n} e^{-\psi_1} \geq \int_{\RR^n} \left( \vphi - \vphi_1 \right) d \mu.
\label{eq_1446}
\end{equation}
\end{enumerate}
Then $\psi$ is essentially-continuous.
\label{prop_1405}
\end{prop}

We remark that the function $\vphi$ in Proposition \ref{prop_1405} is $\mu$-integrable, according to Proposition~\ref{prop_1114}, and hence the right-hand side of (\ref{eq_1446}) is well-defined.
The assumption that $\mu$ is the moment measure of $\psi$  is somewhat redundant;
indeed, property~\eqref{eq_1446} forces $\mu$ to be the moment measure of $\psi$, as we will see in the next section.


\begin{proof}[Proof of  Proposition \ref{prop_1405}.]
Denote
$$K=\{\psi<+\infty\} \quad \textrm{and}\quad S=\partial K.$$
The set $K$ is convex with a non-empty interior. Lower semi-continuity ensures the continuity of $\psi$ at points $x$ where $\psi(x) = +\infty$, so what remains to prove is that $\cH^{n-1} (A)= 0$ where
$$ A:=\{x \in S \; ; \ \psi(x) <+\infty\} .$$
Moreover, in order to prove that $\cH^{n-1} (A) = 0$, it is enough to prove that for every $x\in A$ there exists $U(x)\subset S$, an open neighborhood in $S$ of $x$, such that $\cH^{n-1} (U(x)\cap A) = 0$.
So let us fix a point $x_0 \in A$. The boundary $S = \partial K$ is locally the graph of a convex function in an appropriate coordinate system.
This means that there exists an open  neighborhood $U_0\subset S$ of $x_0$ in $S$, a direction $\theta\in \R^n$, $|\theta|=1$ and $c_0 > 0$ such that
\begin{equation}\label{propsetU}
\forall x \in U_0,\  \forall t >0, \qquad x +t \theta \notin \overline{K}
\end{equation}
where $\overline{K}$ is the closure of $K$, and also
\begin{equation}
\forall x,y \in U_0, \quad \quad \quad  \left| Proj_{\theta^\perp}(x) - Proj_{\theta^\perp}(y) \right| \geq c_0 |x-y|.  \label{eq_1204}
\end{equation}
Here,
$Proj_{\theta^\perp}(x) = x - (x \cdot \theta) \theta$ is the orthogonal projection operator
onto the hyperplane orthogonal to $\theta$. The reader may verify that (\ref{eq_1204}) is equivalent to the fact that
the convex function whose graph is the subset $U_0 \subset S$, is locally Lipschitz.

\medskip
We want to prove that $\cH^{n-1}(U_0\cap A)=0$. The increasing sequence  of sets $A^k$, defined by
$A^k =  \{x \in U_0\; ; \ \psi(x) \le k\}$, tends to  $ U_0 \cap A$ as $k\to +\infty$. So we are done  if we can prove that, for  any $k\ge 1$, we have $\cH^{n-1}(A^k) = 0$. In the sequel we fix some $k_0\ge 1$.

\medskip
 For a small $ \eps >0$ let us introduce
 \begin{equation}
 \vphi_\eps(y) = \max \left \{ \vphi(y), \vphi(y) + \eps (y \cdot \theta) - 1 \right \}
 \quad \quad \quad \quad (y \in \RR^n).
 \label{eq_1625}
 \end{equation}
 The function $\vphi_{\eps}$ is convex, as it is the maximum of two convex functions. It is  also $\mu$-integrable, since $\vphi$ is
$\mu$-integrable (from Proposition \ref{prop_1114}) as well as all linear
functions (from Lemma \ref{lem_1439}), and the maximum of two integrable functions is integrable itself.
Denoting $\psi_{\eps} = \vphi_{\eps}^*$,
the assumption~(\ref{eq_1446}) with $\varphi_1=\varphi_\eps$
rewrites as
\begin{equation}\label{assumptioneps}
\log \int_{\RR^n} e^{-\psi_\eps} - \log \int_{\RR^n} e^{-\psi} \leq \int_{\RR^n} \left( \vphi_\eps - \vphi \right) d \mu.
\end{equation}
 We are going to examine the first order of each side of this inequality as $\eps\to 0$.
To treat the right-hand side, we define $B_\eps = \{y \in \R^n \; ; \ |y|\le 1 / \eps\}$, for which we have
\begin{equation}
 \eps (y \cdot \theta) - 1 \leq 0 \quad \quad \quad \quad \text{for} \ y \in B_\eps .\label{eq_1655} \end{equation}
 We find that
\begin{align} \nonumber
\int_{\RR^n} \big(\vphi_{\eps}  - \vphi \big)\, d \mu  = \int_{\RR^n} \max \left \{0, \eps (y \cdot \theta) - 1 \right \} d \mu(y)
& = \int_{\RR^n \setminus B_\eps} \max \left \{0, \eps (y \cdot \theta) - 1 \right \} d \mu(y)\\ & \leq \eps\, \int_{\RR^n \setminus B_\eps} |y|\, d\mu(y) \label{eq_1743}
 \end{align}
where we used~\eqref{eq_1655} and the trivial bound $ \max \left \{0, \eps (y \cdot \theta) - 1 \right \} \le \eps \, |y|$. Note that Lemma~\ref{lem_1439} and the definition of the image measure guarantee  that
$$\int_{\RR^n} |y|\, d\mu(y) = \int_{\RR^n} |\nabla \psi|\, e^{-\psi} < +\infty.$$
Therefore, $\int_{\RR^n \setminus B_\eps} |y|\, d\mu(y) \rightarrow 0$ as $\eps \rightarrow 0$.
The bound (\ref{eq_1743}) implies that when $\eps \to 0$ we have
\begin{equation}
 \label{eq_1702}
 \int_{\RR^n} \big(\vphi_{\eps}  - \vphi \big)\, d \mu  \le o(\eps).
\end{equation}
Recall that the Legendre transform reverses order. We deduce from (\ref{eq_1625}) that for any $x \in \RR^n$,
\begin{equation}
 \psi_{\eps}(x) \leq \min \left \{ \psi(x), \psi(x - \eps \theta) + 1 \right \} \le \psi(x),
 \label{eq_1627} \end{equation}
 and so we have,  on the set $A^{k_0} =  \{x \in U_0\; ; \ \psi(x) \le k_0\}$, that
$$\forall x \in A^{k_0}, \qquad   \psi_{\eps}(x) \leq k_0 \quad \quad \text{and} \quad \quad \psi_{\eps}(x + \eps \theta) \leq k_0 + 1. $$
Since $\psi_{\eps}$ is convex, then
\begin{equation}
\forall x \in A^{k_0}, \ \forall t \in [0, \eps],\qquad
\psi_{\eps}(x + t \theta) \leq k_0+1. \label{eq_1508} \end{equation}
Define $A^{k_0}_{\eps} = \left \{ x + t \theta \, ; \, x \in A^{k_0}, \ 0 < t \leq \eps \right \}$.
Since $A^{k_0} \subset U_0$, then we may apply Fubini's theorem in view of \eqref{propsetU}, and obtain
\begin{equation} \textrm{Vol}_n(A^{k_0}_\eps) = \eps\, \cH^{n-1} \left(Proj_{\theta^\perp} \left( A^{k_0} \right) \right) \geq \eps \cdot c_0^{n-1} \cdot \cH^{n-1}(A^{k_0})
\label{eq_1211} \end{equation}
where we used (\ref{eq_1204}) in the last passage.
Moreover~\eqref{propsetU} and (\ref{eq_1508}) imply that
\begin{equation}
A_{\eps}^{k_0} \cap K = \emptyset \quad \quad \text{and} \quad \quad \sup_{z \in A^{k_0}_{\eps}} \psi_{\eps}(z) \leq k_0+1.  \label{eq_1630}
\end{equation}
We may now use (\ref{eq_1627}), (\ref{eq_1211}) and (\ref{eq_1630}) to compute that
\begin{align*}
\int_{\RR^n} e^{-\psi_{\eps}} & = \int_{K} e^{-\psi_{\eps}} + \int_{\RR^n \setminus K} e^{-\psi_{\eps}}
\geq \int_{K} e^{-\psi} + \int_{A_{\eps}^{k_0}} e^{-\psi_{\eps}} \\ & \geq \int_{\RR^n} e^{-\psi} + \int_{A^{k_0}_{\eps}} e^{-(k_0+1)}
\geq \int_{\RR^n} e^{-\psi} +   \eps \cdot c_0^{n-1} \cdot \cH^{n-1}(A^{k_0}) \cdot e^{-(k_0+1)}. \nonumber
\end{align*}
By using  the elementary bound  $\log (s + t) \geq \log s + t / (2s)$ for $0< t \le s$, we find, for $\eps$ small enough, that
\begin{equation}
\log \int_{\RR^n} e^{-\psi_{\eps}} \geq \log \int_{\RR^n} e^{-\psi} + c_1\,  \eps\,  \cH^{n-1}(A^{k_0})
\label{eq_1703_}
\end{equation}
where we have set $c_1= c_0^{n-1} e^{-(k_0+1)} / (2\int_{\R^n} e^{-\psi}) = c_0^{n-1} e^{-(k_0+1)} / 2 >0$.
If we go back to our assumption~\eqref{assumptioneps}, we see that we have established, in view of ~\eqref{eq_1702} and~\eqref{eq_1703_}, that
$$c_1\,  \eps\, \cH^{n-1}(A^{k_0}) \le o(\eps)$$
as $\eps\to 0$. This ensures that $\cH^{n-1}(A^{k_0}) = 0$, as desired.
\end{proof}

\begin{rem} {\rm
Formally,  Theorem \ref{thm_prekopa} does not allow to recast the general case of  the Pr\'ekopa inequality, since
the latter inequality holds true for all log-concave functions, not only the essentially-continuous ones (of course by approximation we can assume that we work with finite, thus continuous,  convex functions). And we have just shown that essential-continuity is needed in our statement. What is going on?
The explanation could be that in the case of a non essentially-continuous $\psi_0$, the measure $\overline\mu$ from~(\ref{eq_1050})
seems to no longer be  a Borel measure on $\RR^n$, but rather a distribution of a certain kind. In this case, perhaps the action of $\overline\mu$ on the convex function
$\vphi$ depends also on the behavior of $\vphi$ at infinity, i.e., on the function
$$ \theta \to \lim_{t \rightarrow +\infty} \frac{\vphi(t \theta)}{t} \in \RR \cup \{ + \infty \}. $$
defined on the sphere $S^{n-1} = \left \{ \theta \in \RR^n \, ; \, |\theta|=1 \right \}$.}
\end{rem}


\section{Existence and Uniqueness}
\label{sec4}

In this section we give the proof of Theorem~\ref{thm_1459}. The  statement concerning the uniqueness of $\psi$ up to translation relies on the sub-gradient form of Pr\'ekopa's inequality (Theorem~\ref{thm_prekopa}) and on the characterization of equality cases in Pr\'ekopa's inequality. The existence of $\psi$ relies on the study of the variational problem~\eqref{eq_1401} that was put forward by Berman and Berndtsson \cite{BB} and on Proposition \ref{prop_1405}. We propose a a new treatment for the variational problem
that utilizes the geometry of log-concave measures.

\begin{proof}[Proof of the uniqueness part in Theorem \ref{thm_1459}.] Let $\mu$ be a measure on $\RR^n$ satisfying the assumptions of Theorem \ref{thm_1459}.
Let $\psi_0, \psi_1: \RR^n \rightarrow \RR \cup \{ + \infty \}$ be two essentially-continuous convex functions whose moment measure is $\mu$. Our goal is to show that there exists $x_0 \in \RR^n$ such that
\begin{equation}
\psi_1(x) = \psi_0(x-x_0) \quad \quad \quad \quad \text{for all} \ x \in \RR^n. \label{eq_955}
\end{equation}
Since $\psi_0$ and $\psi_1$ have the same moment measure, then $\int e^{-\psi_0} = \int e^{-\psi_1} \in (0, +\infty)$. Adding the same constant to both functions,
we may assume that $\mu$ is a probability measure. Denote $\vphi_i = \psi_i^* \ (i=0,1)$. According to
Proposition~\ref{prop_1114}, the convex functions $\vphi_0$ and $\vphi_1$ are $\mu$-integrable. Denote
 $$ \vphi_{1/2} :=  \frac{\vphi_0 + \vphi_1}{2}, \quad \quad \psi_{1/2} := \vphi_{1/2}^*= \Big(  \frac{\psi_0^\ast + \psi_1^\ast}2\Big)^\ast. $$
 Then $\vphi_{1/2}$ is a $\mu$-integrable convex function. From the remarks
at the end of Section~\ref{sec2}, this function is finite in the interior of $conv(Supp(\mu))$. Consequently, $\vphi_{1/2}$ is bounded from below by some affine function,
and $\psi_{1/2}$ is not identically
$+\infty$. Since $\psi_0$ is essentially-continuous, then we may apply Theorem \ref{thm_prekopa}
and conclude that
\begin{equation} \log \int_{\RR^n} e^{-\psi_0} -\log\int_{\RR^n} e^{-\psi_{1/2}} \ge  \int_{\RR^n} (\varphi_0- \varphi_{1/2})\, d\mu.
\label{eq_1007} \end{equation}
Since $\psi_1$ is essentially-continuous, then we may apply Theorem \ref{thm_prekopa} again as follows:
\begin{equation} \log \int_{\RR^n} e^{-\psi_1} -\log\int_{\RR^n} e^{-\psi_{1/2}} \ge  \int_{\RR^n} (\varphi_1- \varphi_{1/2})\, d\mu.
\label{eq_1008} \end{equation}
The right-hand side of (\ref{eq_1007}) is a finite number, as well as the right-hand side of (\ref{eq_1008}).
The left-hand side of (\ref{eq_1007}) is therefore in $\RR \cup \{ +\infty \}$, as well as the left-hand side of (\ref{eq_1008}).
We now add (\ref{eq_1007}) and (\ref{eq_1008}) and divide by two, to obtain
$$\frac{\log \int_{\RR^n} e^{-\psi_0} + \log \int_{\RR^n} e^{-\psi_1}}{2} - \log\int_{\RR^n} e^{-\psi_{1/2}} \geq 0.
\label{eq_1010}
$$
But from the Pr\'ekopa inequality (\ref{eq_1659}), the converse inequality holds:
$$ \frac{\log \int_{\RR^n} e^{-\psi_0} + \log \int_{\RR^n} e^{-\psi_1}}{2} \leq \log\int_{\RR^n} e^{-\psi_{1/2}}.
$$
Hence $\{u_0^\ast=\psi_0, u_1^\ast=\psi_1, (\frac12 u_0+\frac12 u_1)^\ast= \psi_{1/2}\}$ is  a case of equality in  Pr\'ekopa's inequality~(\ref{eq_1659}).
It was explained by Dubuc \cite{dubuc} (see Theorem 12 there and the discussion afterwards) that equality in (\ref{eq_1659}) holds if and only if there exists $x_0 \in \RR^n$
 and $c \in \RR$ such that
$$
u_1^*(x) = u_0^*(x - x_0) + c
$$
for almost any $x \in \RR^n$. Since $\psi_0 = u_0^*$ and $\psi_1 = u_1^*$ are lower semi-continuous convex functions, then equality almost everywhere implies  equality pointwise in $\RR^n$. Therefore $\psi_0$ is a translation of $\psi_1$,
up to an additive constant. Since $\int e^{-\psi_0} = \int e^{-\psi_1}$, there is no need for
an additive constant, and (\ref{eq_955}) is proven. \end{proof}

The rest of the this section is devoted to the existence part of Theorem \ref{thm_1459}. It relies on the study of the following variational problem.

\begin{prop}  Let $\mu$ be a  probability measure on $\RR^n$ satisfying the requirements of Theorem \ref{thm_1459}.
For a $\mu$-integrable function $f: \RR^n \rightarrow \RR \cup  \{ +\infty \}$ we set
$$ \cI_{\mu}(f) := \log \int_{\RR^n} e^{-f^*} - \int_{\RR^n} f d \mu. $$
Then there exists a $\mu$-integrable, convex function $\vphi: \RR^n \rightarrow \RR \cup \{ + \infty \}$
with $\int e^{-\varphi^\ast} = 1$ such that
\begin{equation}  \cI_{\mu}(\vphi) = \sup_{f} \cI_{\mu}(f)
\label{eq_1401}
\end{equation}
where the supremum runs over all $\mu$-integrable functions $f: \RR^n \rightarrow \RR \cup \{ + \infty \}$.
\label{prop_1756}
\end{prop}

Before we prove this proposition, let us see how we can deduce the Theorem from it.

\begin{proof}[Proof of the existence part in Theorem \ref{thm_1459}.] We are given a measure $\mu$ on $\RR^n$ satisfying assumptions (i), (ii) and (iii),
and we need to find an essentially-continuous convex function $\psi$ whose moment measure is $\mu$. We may normalize $\mu$ to be a probability measure -- this
amounts to adding a constant to $\psi$. Apply Proposition~\ref{prop_1756}, and conclude that there exists a $\mu$-integrable
convex function $\vphi: \RR^n \rightarrow \RR \cup \{ + \infty \}$, such that, denoting $\psi = \vphi^*$
we have $\int e^{-\psi} = 1$ and also:
\begin{enumerate}
\item[(*)] For any $\mu$-integrable function $\vphi_1: \RR^n \rightarrow \RR \cup \{ + \infty \}$, denoting $\psi_1 = \vphi_1^*$,
\begin{equation}   \log \int_{\RR^n} e^{-\psi} - \log \int_{\RR^n} e^{-\psi_1} \geq \int_{\RR^n} \left( \vphi - \vphi_1 \right) d \mu.
\label{eq_1758}
\end{equation}
\end{enumerate}
Since $\int e^{-\psi} = 1$, then we may use Jensen's inequality, and conclude that for any lower semi-continuous convex function $\psi_1:\RR^n \rightarrow \RR \cup \{ + \infty \}$,
\begin{equation}  \int_{\RR^n} \left( \psi - \psi_1 \right) e^{-\psi} \leq \log \int_{\RR^n} e^{\psi - \psi_1} e^{-\psi} = \log \int_{\RR^n} e^{-\psi_1}.
\label{eq_1759}
\end{equation}
To be more precise, we need to explain why the left-hand side of (\ref{eq_1759}) makes sense as an element in $\RR \cup \{-\infty \}$.
This is because the convex function $\psi_1$ is bounded from below by an affine function, which is integrable
with respect to $e^{-\psi(x)} \,dx$, while the function $\psi e^{-\psi}$ is integrable, as was already shown in the proof of Proposition \ref{prop_1114}.
Next, use (\ref{eq_1758}), (\ref{eq_1759}) and $\int e^{-\psi} = 1$ to arrive at
\begin{enumerate}
\item[(**)] For any $\mu$-integrable  function $\vphi_1: \RR^n \rightarrow \RR \cup \{ + \infty \}$, denoting $\psi_1 = \vphi_1^*$,
\begin{equation}   \int_{\RR^n} \psi\,  e^{-\psi} + \int_{\RR^n} \vphi \, d \mu \leq \int_{\RR^n} \psi_1\,  e^{-\psi} + \int_{\RR^n} \vphi_1\,  d \mu.
\label{eq_1800}
\end{equation}
\end{enumerate}
However, (**) is precisely the Kantorovich dual-variational problem associated with {\it optimal transportation} between the measures $e^{-\psi(x)} \,dx$ and $\mu$, see Brenier \cite{brenier}
or Gangbo and McCann \cite{gangboo_mccann}.
The inequality (\ref{eq_1800}) implies that $\psi$ is a minimizer of this variational problem, and then,
a standard, elementary, argument
implies that $\nabla \psi$ pushes forward the measure $e^{-\psi(x)} \,dx$
to the measure $\mu$ (this property is formally nothing else than the Euler-Lagrange equation for the non-linear variational problem (**)).

 For readers that are not familiar with the theory of optimal transportation, the standard argument we refer to goes roughly as follows: Pick any continuous, compactly-supported
function $b: \RR^n \rightarrow \RR$. Denoting $\vphi_t = \vphi + t b$ and $\psi_t = \vphi_t^*$, one can check that
$$ \left. \frac{d \psi_t(x)}{dt} \right|_{t = 0} = -b(\nabla \psi(x)) $$
at any point $x \in \RR^n$ in which $\psi$ is differentiable (see, e.g., Berman and Berndtsson \cite[Lemma 2.7]{BB} for a short proof).
From the bounded convergence theorem,
\begin{equation}  \left. \frac{d}{dt} \left( \int_{\RR^n} \psi_t e^{-\psi} + \int_{\RR^n} \vphi_t d \mu \right) \right|_{t =0 } = - \int_{\RR^n} b(\nabla \psi(x))\,  e^{-\psi(x)} \,dx + \int_{\RR^n} b\,  d \mu. \label{eq_1405} \end{equation}
However, the expression in (\ref{eq_1405}) must vanish according to (**). Therefore equation~\eqref{eq:defimage} holds for every continuous compactly-supported function.
 This ensures that $\nabla \psi$ pushes the measure $e^{-\psi(x)} \,dx$ forward to the measure $\mu$.

Thus $\mu$ is the moment measure of $\psi$. Finally, property (*) and the fact that $\mu$ is the moment measure of $\psi$ allow us to apply Proposition \ref{prop_1405} and deduce that $\psi$ is essentially-continuous.
\end{proof}

\begin{rem}{\rm \label{rem_maximizer}
We see that there is a strong connection between the variational problem~\eqref{eq_1401} used to construct $\psi$ and the sub-gradient form of Pr\'ekopa's inequality from Theorem~\ref{thm_prekopa}. This calls for several observations:
\begin{itemize}
\item It follows from Pr\'ekopa's inequality that the functional $f\to \cI_{\mu}(f)$ from Proposition~\ref{prop_1756} is concave (and strictly concave modulo addition of affine maps) on the convex set of convex function $f$ that are finite in a neighborhood of $0$. It is therefore not surprising that it admits a maximum.
\item
With the notation of Proposition~\ref{prop_1756}, if we know that $\mu$ is the moment measure of $\psi_0$, then Theorem~\ref{thm_prekopa} forces $\varphi_0=\psi_0^\ast$ to be a maximizer in~\eqref{eq_1401}.
\end{itemize}}
\end{rem}

It remains to prove Proposition~\ref{prop_1756}. This requires several steps that are detailed in the next Lemmas. The proof of the Proposition is given at the end of this section, once these lemmas are established. Recall that we denote $S^{n-1} = \{\theta\in \R^n\; ;\  |\theta| = 1\}$.

\begin{lemma} Let $\mu$ be a probability measure on $\RR^n$ satisfying the assumptions of Theorem \ref{thm_1459}. Then there exists $c_{\mu} > 0$
such that for any $\mu$-integrable, convex function $\vphi: \RR^n \rightarrow \RR \cup \{ + \infty \}$ with $\exp(-\vphi)$ integrable and $\vphi(0) = 0$, we have
\begin{equation}  \left( \int_{\RR^n} e^{-\vphi} \right)^{1/n} \left(\int_{\RR^n} \left[\vphi - \inf \vphi \right] d \mu  + 1 \right) \geq c_{\mu}.
\label{eq_1513}
\end{equation}
\label{lem_2116}
\end{lemma}

\begin{proof} From the dominated convergence theorem, the function
$$ S^{n-1} \ni \theta \to \int_{\RR^n} |x \cdot \theta| d \mu(x) $$
is continuous.
Since the support of $\mu$ is not contained in a hyperplane, then this function is always positive.
 Hence its infimum on the sphere, denoted by $m_{\mu}$, is positive.
We will prove (\ref{eq_1513}) with
$$ c_{\mu} = \frac{\kappa_n^{1/n} m_{\mu}}{4 e^{1/n}}  $$
where $\kappa_n$ is the volume of the $n$-dimensional unit ball.
Let $\vphi$ be a convex function on $\RR^n$ with $\exp(-\vphi)$ integrable and $\vphi(0) = 0$.
Set
$$ K = \{ x \in \RR^n ; \vphi(x) \leq 1 \}. $$
Then $K$ is a convex set, containing the origin in its interior. Let us denote by $r_\varphi$ the radius of the largest Euclidean ball contained in the centrally-symmetric convex set $$ K-K = \left \{ x - y \, ; \, x,y \in K \right \}. $$
Since $\exp(-\vphi)$ is integrable, then $K$ has finite
volume. According to the Rogers-Shepherd inequality \cite{RS}, the set $K - K$
has volume at most $4^n Vol_n(K)$. In particular,
$K - K$ cannot contain an Euclidean ball of radius greater than
$$ R_{\vphi} := 4 \kappa_n^{-1/n} Vol_n(K)^{1/n}, $$
which means that  $r_\varphi \le R_\varphi$.
By definition of $r_\varphi$,  we may find a vector   $\theta_0 \in S^{n-1}$ such that
\begin{equation}  \sup_{x \in K-K} |x \cdot \theta_0| = \sup_{x \in K-K} x \cdot \theta_0 = r_{\vphi}\le R_\varphi. \label{eq_1706} \end{equation}
Since $K \subseteq K-K$, we conclude from (\ref{eq_1706}) that for any $x \in \RR^n$:
\begin{equation}  |x \cdot \theta_0| \ge R_{\vphi} \quad \Longrightarrow \quad
\vphi(x) \geq 1. \label{eq_1546} \end{equation}
By the convexity of $\vphi$, we have for any $x \in \RR^n$ with  $|x \cdot \theta_0| \geq R_{\vphi}$,
\begin{equation}  \vphi \left( \frac{R_{\vphi}}{|x \cdot \theta_0|} x \right) \leq \frac{R_{\vphi}}{|x \cdot \theta_0|} \vphi(x) + \left( 1 - \frac{R_{\vphi}}{|x \cdot \theta_0|} \right) \vphi(0)
= \frac{R_{\vphi}}{|x \cdot \theta_0|} \vphi(x).
\label{eq_1547} \end{equation}
Now (\ref{eq_1546}) and (\ref{eq_1547}) yield  that for any $x \in \RR^n$ with  $|x \cdot \theta_0| \geq R_{\vphi}$,
\begin{equation}  \vphi(x) \geq \frac{|x \cdot \theta_0|}{R_{\vphi}} + \left( \inf \vphi - 1 \right),
\label{eq_1550}
\end{equation}
since $\inf \vphi \leq 0$.
However, (\ref{eq_1550}) holds trivially when $|x \cdot \theta_0| < R_{\vphi}$, and therefore the bound (\ref{eq_1550}) is valid
for any $x \in \RR^n$. Integrating it, we find
\begin{equation}  \int_{\RR^n} \vphi d \mu \geq  \frac{1}{R_{\vphi}} \int_{\RR^n} |x \cdot \theta_0| d \mu(x) + (\inf \vphi - 1) \geq \frac{m_{\mu}}{R_{\vphi}} + (\inf \vphi - 1). \label{eq_1558}
\end{equation}
However,
\begin{equation}
\int_{\RR^n}  e^{-\vphi} \geq \int_{K} e^{-\vphi} \geq Vol_n(K) / e = \frac{\kappa_n}{4^n e}   R_{\vphi}^n. \label{eq_1644}
\end{equation}
From (\ref{eq_1558}) and (\ref{eq_1644}) we deduce (\ref{eq_1513}). \end{proof}

\begin{lemma} Let $\mu$ be a  probability measure on $\RR^n$ satisfying the requirements of Theorem \ref{thm_1459}.
Let $\vphi: \RR^n \rightarrow \RR \cup \{+\infty \}$ be a $\mu$-integrable convex function
with $\vphi(0) = 0$ and denote $\psi = \vphi^*$.
Then,
\begin{equation}
 \int_{\RR^n} \vphi \, d \mu \geq \frac{c_{\mu}}{2 \pi} \left( \int_{\RR^n} e^{-\psi} \right)^{1/n} - (n+1) \label{eq_1111}
\end{equation}
where $c_{\mu} > 0$ is the constant from Lemma \ref{lem_2116}. \label{lem_1119}
\end{lemma}

\begin{proof} From the remarks
at the end of Section~\ref{sec2},
the function $\vphi$
is finite near the origin, and $\exp(-\psi)$ is an integrable, log-concave function.
Additionally, since $\vphi$ is finite near the origin, there exists $x_0 \in \partial \vphi(0)$.
Then $\vphi(y) \geq x_0 \cdot y$ for all $y \in \RR^n$, and
$$ \int_{\RR^n} \vphi\,  d \mu \geq \int_{\RR^n} (x_0 \cdot y) d \mu(y) = 0. $$
In proving (\ref{eq_1111}), we may thus restrict attention to the case where $\int \exp(-\psi) > 0$.
Furthermore, adding a linear function to $\vphi$ corresponds to translating $\psi$, and does not
change neither the left-hand side of (\ref{eq_1111}) nor the right-hand side. We may thus
translate translate $\psi$ so that
\begin{equation} \int_{\RR^n} x_i e^{-\psi(x)} \,dx =0 \quad \quad \quad \quad (i=1,\ldots,n),
\label{eq_1724}
\end{equation}
i.e. the barycenter of $e^{-\psi(x)} \,dx$ lies at the origin.
Since $\vphi(0) = 0$ then $\inf \psi = 0$.
 An inequality
proven in Fradelizi \cite{fradelizi} states that, thanks to (\ref{eq_1724}),
$$  \psi(0) \leq \inf_{x \in \RR^n} \psi(x) + n = n. $$
Consequently,
$$ \inf_{y \in \RR^n} \vphi(y) = - \psi(0) \geq -n. $$
Since $\vphi(0) = 0$, then we may apply Lemma
\ref{lem_2116}, and conclude that
\begin{equation}
 \left( \int_{\RR^n} e^{-\vphi} \right)^{1/n} \left(\int_{\RR^n} \vphi d \mu + n + 1 \right) \geq c_{\mu}. \label{eq_2118}
 \end{equation}
Since the log-concave function $e^{-\psi}$ has barycenter at the origin, the functional Santal\'o inequality from Artstein, Klartag and Milman \cite{AKM} asserts that
\begin{equation} \int_{\RR^n} e^{-\psi} \int_{\RR^n} e^{-\vphi} \leq (2 \pi)^n.
\label{eq_2125}
\end{equation}
Now (\ref{eq_1111}) follows from (\ref{eq_2118}) and (\ref{eq_2125}), and the lemma is proven. \end{proof}

\begin{lemma} Let $\mu$ be a finite Borel measure in $\RR^n$ and let $K$ be the interior
of $conv(Supp(\mu))$. If $x_0 \in K$, then there exists $C_{\mu, x_0} > 0$
with the following property: For any non-negative, $\mu$-integrable,
convex function $\vphi: \RR^n \rightarrow \RR \cup \{ + \infty \}$,
$$ \vphi(x_0) \leq C_{\mu, x_0} \int_{\RR^n} \vphi d \mu. $$ \label{lem_2211}
\end{lemma}

\begin{proof} Since $x_0$ is in the interior of $conv(Supp(\mu))$, then
for any $\theta \in S^{n-1}$,
\begin{equation}  \mu \left( \left \{ x \in \RR^n \, ; \, (x - x_0) \cdot \theta > 0 \right\} \right)  > 0.
\label{eq_2152} \end{equation}
By Fatou's lemma, the left-hand side of (\ref{eq_2152}) is a lower semi-continuous function of $\theta \in S^{n-1}$.
Hence the infimum of the left-hand side of (\ref{eq_2152}) over $\theta \in S^{n-1}$, denoted by $m_{\mu, x_0}$,
is attained and is therefore positive. Let $\vphi$ be a non-negative, $\mu$-integrable, convex
function. The function $\vphi$ is necessarily finite near $x_0$, and hence  there exists $y_0 \in \partial \vphi(x_0)$.
Using that $\vphi(x) \geq \vphi(x_0) + y_0 \cdot (x - x_0)$ for all $x$, we find
$$ \int_{\RR^n} \vphi \, d \mu  \geq \int_{ \{x ; (x - x_0) \cdot y_0 \geq 0 \}} \vphi(x) d \mu(x) \geq
\vphi(x_0) \cdot \mu \Big( \Big\{ x ;  (x - x_0) \cdot  y_0 \geq 0 \Big\} \Big) \geq m_{\mu, x_0} \cdot \vphi(x_0). $$
The lemma follows with $C_{\mu, x_0} = 1 / m_{\mu, x_0}$. \end{proof}

\begin{lemma} Let $\mu$ be a measure on $\RR^n$ satisfying the requirements of Theorem \ref{thm_1459}.
Assume that with any $\ell \geq 1$ we are given a $\mu$-integrable, non-negative convex function $\vphi_\ell: \RR^n \rightarrow [0,+\infty]$
with $\vphi_{\ell}(0) = 0$ and such that
\begin{equation} \sup_\ell \int_{\RR^n} \vphi_\ell\,  d \mu < +\infty.
\label{eq_1520}
\end{equation}
Then there exists a subsequence $\{ \vphi_{\ell_j} \}_{j=1,2,\ldots}$ and a non-negative, $\mu$-integrable, convex function $\vphi: \RR^n \rightarrow \RR \cup \{ + \infty \}$ such that,
denoting  $\psi_{\ell} = \vphi_{\ell}^*$ and $\psi = \vphi^*$,
\begin{equation}  \int_{\RR^n} \vphi d \mu \leq \liminf_{j \rightarrow \infty} \int_{\RR^n} \vphi_{\ell_j} d \mu \quad \quad \text{and} \quad \quad \int_{\RR^n} e^{-\psi} \geq \limsup_{j \rightarrow \infty} \int_{\RR^n} e^{-\psi_{\ell_j}}. \label{eq_1531} \end{equation} \label{lem_1600}
\end{lemma}

\begin{proof} Denote by $K$ the interior of $conv(Supp(\mu))$, which is an open convex set, containing $0$.
 From (\ref{eq_1520}) and Lemma \ref{lem_2211}  we have that for any $x \in K$,
$$ \sup_\ell \vphi_\ell(x) < +\infty. $$
According to Rockafellar \cite[Theorem 10.9]{rockafellar},
there exists a subsequence $\{ \vphi_{\ell_j} \}_{j=1,2,\ldots}$ that converges pointwise in $K$ to a convex function $\vphi: K \rightarrow \RR$.
The convex function $\vphi$ is finite and thus continuous on the open set $K$. Additionally, $\vphi$  is non-negative in $K$ and achieves its minimum at the origin, where it vanishes. We extend the definition of $\vphi$ by setting
$\vphi(x) = +\infty$ for $x \not \in \overline{K}$.
We still need to define $\vphi(x)$ for points $x \in \partial K$. We will set for such $x \in \partial K$,
$$ \vphi(x) := \lim_{\lambda \rightarrow 1^-} \vphi(\lambda x). $$
This limit always exists in $[0, +\infty]$, since the function $\lambda \to \vphi(\lambda x)$ is non-decreasing for $\lambda \in (0,1)$.
Moreover, we have that $\vphi(\lambda x) \nearrow \vphi(x)$ as $\lambda \rightarrow 1^-$ for any $x \in \overline{K}$.
The resulting function $\vphi: \RR^n \rightarrow \RR \cup \{ + \infty \}$ is therefore convex and
non-negative with $\vphi(0) = 0$. We need to show that $\vphi$ is $\mu$-integrable and satisfies (\ref{eq_1531}). To that end, pick $0 < \lambda <1$.
For any $x \in \overline{K}$ the point $\lambda x$ belongs to $K$, and by the pointwise convergence in $K$,
\begin{equation}  \vphi(\lambda x) = \lim_{j \rightarrow \infty} \vphi_{\ell_j}(\lambda x) \quad \quad \quad \quad (x \in \overline{K}).
\label{eq_1414} \end{equation}
Since $Supp(\mu) \subseteq \overline{K}$, then from (\ref{eq_1414}) and Fatou's lemma, for any $0 < \lambda < 1$,
\begin{equation}  \int_{\RR^n} \vphi(\lambda x) d \mu(x) \leq \liminf_{j \rightarrow \infty} \int_{\RR^n} \vphi_{\ell_j}(\lambda x) d \mu(x) \leq \liminf_{j \rightarrow \infty} \int_{\RR^n} \vphi_{\ell_j}(x) d \mu(x) < +\infty \label{eq_1329} \end{equation}
where we used the fact that $\vphi_{\ell_j}(\lambda x) \leq \lambda \vphi_{\ell_j}(x) \leq \vphi_{\ell_j}(x)$, by convexity.
Recall that we have $\vphi(\lambda x) \nearrow \vphi(x)$ as $\lambda \rightarrow 1^-$ for any $x \in \overline{K}$.
From the monotone convergence theorem and (\ref{eq_1329}),
\begin{equation}  \int_{\RR^n} \vphi(x) d \mu(x) = \lim_{\lambda \rightarrow 1^-} \int_{\RR^n} \vphi(\lambda x) d \mu(x) \leq \liminf_{j \rightarrow \infty} \int_{\RR^n} \vphi_{\ell_j}(x) d \mu(x)
< +\infty.
\label{eq_1421} \end{equation}
This completes the proof of the first part of (\ref{eq_1531}). It still remains to prove the second part of (\ref{eq_1531}).
The function $\vphi$ is $\mu$-integrable, hence finite near the origin. Therefore $\exp(-\psi)$ is integrable, for $\psi = \vphi^*$.  Let $x_1,x_2,\ldots$ be a dense sequence in $K$. Then for any $y \in \RR^n$,
$$ \psi(y) = \sup_{x \in \RR^n} \left[ x \cdot y - \vphi(x) \right] = \sup_{x \in K} \left[ x \cdot y - \vphi(x) \right]
= \sup_{i\ge 1} \left[ x_i \cdot y - \vphi(x_i) \right] $$
by the continuity of $\vphi$ in $K$. For $j \geq 1$, set $\displaystyle \tilde{\psi}_j(x) = \max_{1\le i\le j} \left[ x_i \cdot y - \vphi(x_i) \right]$.
Now, for a sufficiently large $j$, the set $conv(x_1,\ldots,x_j)$ contains the origin in its interior and hence $\exp(-\tilde{\psi}_j)$ is integrable.
Since $\tilde{\psi}_j \nearrow \psi$ then from the monotone convergence theorem,
$$ \int_{\RR^n} e^{-\psi} = \lim_{j \rightarrow \infty} \int_{\RR^n} e^{-\tilde{\psi}_j}. $$
Fix $\eps > 0$. Then there exists $j_0$ such that $\int \exp(-\tilde{\psi}_{j_0})$ deviates from $\int \exp(-\psi)$ by at most $\eps$.
Abbreviate $\tilde{\psi} = \tilde{\psi}_{j_0}$. Since $\vphi_{\ell_j} \rightarrow \vphi$ pointwise on the set $\{x_1,\ldots,x_{j_0} \}$, then for sufficiently large $j$,
\begin{equation}  \psi_{\ell_j}(x) \geq \tilde{\psi}(x) - \eps \quad \quad \quad \quad \text{for all} \ x \in \RR^n. \label{eq_1608} \end{equation}
From the definition of $j_0$ and from (\ref{eq_1608}),
$$ \int_{\RR^n} e^{-\psi} \geq \int_{\RR^n} e^{-\tilde{\psi}} - \eps \geq  - \eps + e^{-\eps} \limsup_{j\rightarrow \infty} \int_{\RR^n} e^{-\psi_{\ell_j}}.
$$
Since $\eps > 0$ was arbitrary, then the second part of (\ref{eq_1531}) follows.
\end{proof}

We now have all  the ingredients for the proof of the Proposition.

\begin{proof}[Proof of Proposition~\ref{prop_1756}] Set $\tilde{c}_{\mu} = \cI_\mu(\tilde f)$ for $\tilde f(x) = |x|$. Then $\tilde{c}_{\mu}$
is some finite real number. Let $f_1,f_2,\ldots$ be a maximizing sequence of $\mu$-integrable functions, i.e.,
$$ \cI_{\mu}(f_\ell) \stackrel{\ell \rightarrow \infty} \longrightarrow \sup_{f} \cI_{\mu}(f) \ge \tilde c_\mu.$$
Let $M_\ell > 0$ be a sufficiently large number so  that $\int |f_\ell| 1_{ \{ f_\ell \leq -M_\ell \}}\, d \mu \leq 1/\ell$.
Denote $g_\ell = \max \{ f_\ell, -M_\ell \}\ge f_\ell$.
We have that $g_\ell$ is $\mu$-integrable with
\begin{equation} 0\le \int (g_\ell-f_\ell)\, d\mu = \int_{ \left \{ f_\ell\le -M_\ell \right \}} (-M_\ell-f_\ell)\, d\mu \le  -\int_{ \left \{ f_\ell\le -M_\ell \right \}} f_\ell \, d\mu \le \frac1\ell. \label{eq_458} \end{equation}
Since $\exp(-g_\ell^\ast)\ge \exp(-f_\ell^\ast)$ pointwise then from (\ref{eq_458}),
$$ \cI_{\mu}(g_\ell) \geq \cI_{\mu}(f_\ell) - \frac{1}{\ell} \stackrel{\ell \rightarrow \infty} \longrightarrow \sup_{f} \cI_{\mu}(f). $$
Furthermore, the function $\vphi_\ell = (g_\ell^*)^*: \RR^n \rightarrow \RR \cup \{ + \infty \}$ is convex, lower semi-continuous,
 pointwise smaller
than $g_\ell$, and it is always at least $-M_\ell$. In particular $\vphi_\ell$ is $\mu$-integrable.
Note also that $\vphi_\ell^* = g_\ell^*$, hence
$$ \cI_{\mu}(\vphi_\ell) \geq \cI_{\mu}(g_\ell) \stackrel{\ell \rightarrow \infty} \longrightarrow \sup_{f } \cI_{\mu}(f) \geq \tilde{c}_{\mu}. $$
Adding an affine function to $\vphi_{\ell}$ does not change $\cI_\mu(\vphi_{\ell})$.
We may therefore add an affine function and assume that $\vphi_\ell(0) = \inf \vphi_\ell = 0$ for all $\ell$. We arrived at a sequence $(\varphi_\ell)$ of nonnegative, $\mu$-integrable, convex functions such that  $\cI_{\mu}(\varphi_\ell) \to  \sup_{f} \cI_{\mu}(f) $.
 Furthermore, we may remove finitely many elements from the sequence $\{ \vphi_{\ell} \}$ and assume that for all $\ell$,
\begin{equation}
\cI_{\mu}(\vphi_\ell) \geq \tilde{c}_{\mu} - 1. \label{eq_1148}
\end{equation}
Lemma \ref{lem_1119} implies that for any $\ell$,
\begin{equation}
\log \int_{\RR^n} e^{-\vphi_\ell^*} - \cI_{\mu}(\vphi_\ell) \geq \frac{c_{\mu}}{2 \pi} \left( \int_{\RR^n} e^{-\vphi_\ell^*} \right)^{1/n} - (n+1).
\label{eq_1136}
\end{equation}
Combining (\ref{eq_1148}) and (\ref{eq_1136}) with the fact that $\log(t)=o(t^{1/n})$ when $t\to +\infty$, we conclude that
\begin{equation}  \sup_\ell  \int_{\RR^n} e^{-\vphi_\ell^*} < +\infty. \label{eq_1346}
\end{equation}
Consequently, $\sup_{f} \cI_{\mu}(f) \in \RR$. Moreover, from (\ref{eq_1148}) and (\ref{eq_1346}) we have that
\begin{equation} \sup_\ell \int_{\RR^n} \vphi_\ell\,  d \mu < +\infty.
\label{eq_1614} \end{equation}
We may apply Lemma \ref{lem_1600} based on (\ref{eq_1614}), and conclude that there exists a subsequence $\{ \vphi_{\ell_j} \}_{j=1,2,\ldots}$ and a non-negative, $\mu$-integrable  convex function $\vphi: \RR^n \rightarrow \RR \cup \{ + \infty \}$ such that
$$ \cI_{\mu}(\vphi) \geq \limsup_{j \rightarrow \infty} \cI_{\mu} \left(\vphi_{\ell_j} \right) = \sup_{f} \cI_{\mu}(f) \in \RR $$
where the supremum runs over all $\mu$-integrable functions $f: \RR^n \rightarrow \RR \cup \{ + \infty \}$. So $\cI_{\mu}(\vphi) =  \sup_{f} \cI_{\mu}(f)$. Moreover, since $\int \vphi d \mu
\in \RR$ and $\cI_{\mu}(\vphi) \in \RR$, we have that $\int e^{-\psi} \in (0, \infty)$. Adding a constant to $\vphi$ does not
change $\cI_{\mu}(\vphi)$. Therefore we may normalize $\vphi$ by adding a constant and arrange that $\int e^{-\psi} = 1$, as announced.

\end{proof}

\section{Problems of a similar nature}
\label{sec5}

Theorem \ref{thm_1459} is analogous to several results
and problems in convex geometry.
The closest problem is certainly the {\it logarithmic Minkowski problem} of
B\"or\"oczky, Lutwak, Yang and Zhang \cite{BLYZ}. It is concerned with associating to a Borel measure on a sphere, a convex body having this measure as its \emph{cone measure} (see below). In some sense, the moment measure of a convex function is the functional analogue of the cone measure of a convex body. More precisely, given a convex body, we can reproduce its cone measure from a suitable moment measure. Indeed, let $K \subset \RR^n$ be a convex body containing the origin. The Minkowski functional of $K$ is
$$ \| x \|_K = \inf \left \{ \lambda > 0 ; x \in \lambda K \right \} \quad \quad \quad \quad (x \in \RR^n). $$
Suppose that $f: [0, \infty) \rightarrow \RR \cup \{+ \infty \}$ is convex, increasing and non-constant. Set
$$
 \psi(x) = f(\| x \|_K) \quad \quad \quad \quad (x \in \RR^n). $$
Then $\psi$ is a convex function on $\RR^n$ with $0 < \int \exp(-\psi) < +\infty$, and almost everywhere in $\RR^n$,
$$ \nabla \psi(x) = f^{\prime}(\| x \|_K) \nabla \| x \|_K. $$
Integrating in polar coordinates, one may verify that the moment measure $\mu$ of $\psi$ takes
the  following form:
For any Borel subsets $A \subseteq \partial K^{\circ}$ and $B \subseteq [0, \infty)$,
$$ \mu(A \times B) = \nu_1(A) \nu_2(B). $$
The measure $\nu_2$ is not very important and it depends solely on the choice of $f$:
it is just the push-forward of the measure
on $[0, \infty)$ with density
$$ n t^{n-1} e^{-f(t)} $$
under the map $t \to f^{\prime}(t)$.
The geometry of the construction is in the measure  $\nu_1$, which  is a measure on $\partial K^{\circ}$ that does not depend on $f$. In fact,
$\nu_1$ is the push-forward of the Lebesgue measure on $K$ via the \emph{un-normalized} Gauss map
\begin{eqnarray*}
K & \to &\partial K^{\circ}\\
x &\to &\nabla \| x \|_K .
\end{eqnarray*}
For $0 \neq x \in \RR^n$ denote $\cR(x) = x / |x|$. Then the measure
$$ \cR_*(\nu_1) $$
is referred to as the {\it cone volume measure of $K$} in \cite{BLYZ},
where $\cR_*(\nu_1)$ is the push-forward of $\nu_1$ via $\cR$. This shows that the cone measure of a convex body can be recovered from the moment measure of a particular  convex function. This suggests that Theorem~\ref{thm_1459} gives the solution to a functional extension of the logarithmic Minkwoski problem, although, unfortunately, we do not see a quick way to recover the geometric form from it.

\medskip The original Minkowski problem from 1897 is related to the surface area measure
and not to the volume measure. The expression
\begin{equation}  \int_{\RR^n} |\nabla \psi| e^{-\psi} \label{eq_1213}
\end{equation}
is sometimes viewed as the analog, for a log concave function $e^{-\psi}$,  of the concept of a surface area of a convex body.
According to Lemma \ref{lem_1439}, the expression in (\ref{eq_1213}) is finite
whenever $\int \exp(-\psi) < +\infty$. We may therefore push-forward
the measure $|\nabla \psi(x)| \exp(-\psi(x)) \,dx$ under the map $x \to \nabla \psi(x)$.
The resulting measure, denoted by $\nu$, is a simple variant of the moment measure $\mu$ of $\psi$.
Namely,
$$ \frac{d \nu}{d \mu}(x) = |x| \quad \quad \quad \quad (x \in \RR^n). $$
We can therefore apply Theorem \ref{thm_1459} and understand exactly which measures $\nu$
arise this way, from an essentially-continuous convex function $\psi$, and we may also recover $\psi$ from $\nu$, up to translation.
Note that the latter problem is not {\it linearly invariant}. When dealing with the moment measure, on the other hand,
we require
nothing more than the structure of a finite-dimensional linear space.

\medskip There are many other variants of Theorem \ref{thm_1459} that could be interesting. For instance,
the use of the exponential function is convenient, but certainly not crucial. For various functions
$s_1, s_2$ on the real line, one may consider the measure $$ s_1(\psi(x)) \,dx
$$
on $\RR^n$, and push it forward using the map $x \to s_2(\psi(x)) \nabla \psi(x)$.  Perhaps
one has to replace the use of (\ref{eq_2125}) with the functional versions
of Santal\'o's inequality from Fradelizi and Meyer \cite{FM}. We do not investigate these potential
generalizations of Theorem \ref{thm_1459} in this paper.

{
}


\begin{thebibliography}{99}

\bibitem{abreu}
 Abreu, M., {\it  K\"ahler geometry of toric manifolds in symplectic coordinates.} {\it In} Symplectic and contact topology: interactions and perspectives, pp. 1--24, Fields Inst. Commun. 35, A.M.S., 2003.

\bibitem{AKM}
 Artstein-Avidan, S., Klartag, B., Milman, V., {\it  The Santal\'o point of a function, and a functional form of the Santal\'o inequality}.
  Mathematika 51, no. 1-2, (2004), 33–-48.

\bibitem{BB} Berman, R. J., Berndtsson, B., {\it Real Monge-Amp\`ere equations and K\"ahler-Ricci solitons on toric log Fano varieties.}
Preprint, arXiv:1207.6128

\bibitem{brenier}  Brenier, Y., {\it Polar factorization and monotone rearrangement of vector-valued functions. }
Comm. Pure Appl. Math. 44, no. 4, (1991), 375–-417.

\bibitem{BD} Bunch, R. S., Donaldson, S. K., {\it Numerical approximations to extremal metrics on toric surfaces.}
Handbook of geometric analysis. Adv. Lect. Math. (ALM), Vol. 7, No. 1, Int. Press, Somerville, MA, (2008), 1-–28.

\bibitem{BLYZ} B\"or\"oczky, K. J., Lutwak, E., Yang D., Zhang, G.,
{\it The logarithmic Minkowski problem}. To appear in J. Amer. Math. Soc.

\bibitem{cordero} Cordero-Erausquin, D., {\it Some applications of mass transport to Gaussian-type inequalities. }
Arch. Ration. Mech. Anal., Vol. 161, no. 3, (2002), 257–-269.

\bibitem{cordero_klartag}
Cordero-Erausquin, D., Klartag, B.,
{\it Interpolation, convexity and geometric inequalities}.  Geometric Aspects of Functional Analysis, Lecture Notes in Math., Vol. 2050, Springer, (2012), 151--168.


\bibitem{donaldson} Donaldson, S. K., {\it K\"ahler geometry on toric manifolds, and some other manifolds with
large symmetry.} Handbook of geometric analysis. Adv. Lect. Math. (ALM), Vol. 7, No. 1, Int. Press, Somerville, MA, (2008), 29-–75.

\bibitem{physicists} Doran, C., Headrick, M., Herzog, C., Kantor, J., Wiseman, T., {\it
Numerical K\"ahler-Einstein metric on the third del Pezzo.}
Comm. Math. Phys., Vol. 282, No. 2, (2008), 357-–393.


\bibitem{dubuc} Dubuc, S., {\it Crit\`eres de convexit\'e et in\'egalit\'es int\'egrales.}
Ann. Inst. Fourier (Grenoble), Vol. 27, no. 1, (1977), 135–-165.

\bibitem{EG}  Evans, L. C., Gariepy, R. F., {\it Measure theory and fine properties of functions. }
Studies in Advanced Mathematics. CRC Press, Boca Raton, FL, 1992.

\bibitem{fradelizi} Fradelizi, M., {\it Sections of convex bodies through their centroid. }
Arch. Math., Vol. 69, No. 6, (1997), 515–-522.

\bibitem{FM} Fradelizi, M., Meyer, M., {\it Some functional forms of Blaschke-Santal\'o inequality}. Math.
Z., Vol. 256, No. 2, (2007) 379–-395.

\bibitem{gangboo_mccann}  Gangbo, W., McCann, R. J.,
{\it Optimal maps in Monge's mass transport problem. }
C. R. Acad. Sci. Paris S\'er. I Math., Vol. 321, No. 12, (1995), 1653–-1658.

\bibitem{gromov} Gromov, M., {\it Convex sets and K\"ahler manifolds. } Advances in
differential geometry and topology, World Sci. Publ., Teaneck, NJ,
(1990), 1–-38.

\bibitem{K_psi} Klartag, B., {\it Uniform almost sub-gaussian estimates for linear functionals on
convex sets.} Algebra i Analiz (St. Petersburg Math. Journal), Vol. 19, No. 1, (2007), 109–-148.

\bibitem{k_poincare} Klartag, B., {\it Poincar\'e inequalities and moment maps.}
To Appear in Ann. Fac. Sci. Toulouse Math. arXiv:1104.2791


\bibitem{e_legednre} Legendre, E, {\it Toric K\"ahler-Einstein metrics and convex compact polytopes.}
Preprint, arXiv:1112.3239

\bibitem{maurey}
Maurey, B.,
{\it In\'egalit\'e de Brunn-Minkowski-Lusternik, et autres in\'egalit\'es g\'eom\'etriques et fonctionnelles.}
S\'eminaire Bourbaki. Vol. 2003/2004. Ast\'erisque, No. 299, (2005), 95--113.

\bibitem{mccann95}
McCann, R. J., {\it  Existence and uniqueness of monotone measure-preserving maps.}
Duke Math. J.,  Vol. 80, (1995), 309--323.

\bibitem{mccann}  McCann, R. J., {\it A convexity principle for interacting gases.} Adv. Math., Vol. 128, (1997),
153–-179.

\bibitem{rockafellar}  Rockafellar, R. T., {\it Convex analysis. }
Princeton Mathematical Series, No. 28. Princeton University Press, Princeton, NJ, 1970.


\bibitem{RS} Rogers, C. A., Shephard, G. C., {\it The difference body of a convex body.}
Arch. Math., Vol. 8, (1957), 220–-233.

\bibitem{schneider}  Schneider, R., {\it Convex bodies: the Brunn-Minkowski theory.}
 Encyclopedia of Mathematics and its Applications, Vol. 44. Cambridge University Press, Cambridge, 1993.


\bibitem{SK}  Stein, E. M., Shakarchi, R.,
{\it Real analysis. Measure theory, integration, and Hilbert spaces.}
 Princeton Lectures in Analysis, III. Princeton University Press, Princeton, NJ, 2005.

\bibitem{WZ} Wang, X.-J., Zhu, X.,
{\it K\"ahler–-Ricci solitons on toric manifolds with positive first Chern
class.} Advances in Math., Vol. 188, (2004), 87–-103.


\end{thebibliography}
\end{document}